\documentclass[hidelinks,a4paper]{article}

\usepackage[T1]{fontenc}
\usepackage[utf8]{inputenc}
\usepackage{lmodern}
\usepackage{microtype}

\usepackage[a4paper,left = 1.1in,right = 1.1in, top = 1.25in, bottom = 1.25in]{geometry}

\usepackage{amsmath,amsbsy,amssymb,amsthm,amsfonts}
\usepackage{mathrsfs}
\usepackage{mathtools}
\usepackage{enumerate}
\usepackage{float}
\usepackage[dvipsnames]{xcolor}
\usepackage{dsfont}
\usepackage{soul}

\usepackage{multirow}
\usepackage{empheq}
\usepackage{listings}
\usepackage{color}
\usepackage{subcaption}
\usepackage[symbol]{footmisc}
\usepackage{fancyhdr}

\usepackage{graphicx} 
\usepackage[font=small,labelfont=bf]{caption}
\usepackage[square,comma,numbers]{natbib}

\usepackage{hyperref}
\usepackage[linesnumbered, vlined]{algorithm2e}
\usepackage{etoolbox}

\usepackage{bm}
\usepackage{dsfont}
\usepackage{cancel}
\usepackage{amssymb}
\usepackage{commath}
\usepackage{mathtools}
\usepackage{float}
\usepackage{accents}
\counterwithin{equation}{section}

\usepackage{caption}
\usepackage{subcaption}
\usepackage{tcolorbox}
\usepackage{tikz}
\usetikzlibrary{matrix}
\usepackage{cases}

\usepackage{xcolor}
\usepackage{hyperref}
\hypersetup{colorlinks,linkcolor=blue,citecolor=blue,urlcolor = violet}

\usepackage[capitalise]{cleveref}

\def\be{\begin{equation}}
\def\ee{\end{equation}}

\def\bea{\begin{align}}
\def\eea{\end{align}}

\def\bea*{\begin{align*}}
\def\eea*{\end{align*}}

\def\d{\mathrm{d}}

\def\om{\omega}
\def\Om{\Omega}

\def\mo{\mathfrak{o}}
\def\pmo{\partial_\mo}
\def\dmo{\delta_\mo}
\def\vt{\vartheta}

\def\uv{u}
\def\ov{w}
\def\uvgb{\uv^\gamma_\beta}
\def\ovgb{\ov_\gamma^\beta}
\def\uvg{\uv^\gamma}
\def\ovg{\ov_\gamma}

\def\vep{\varepsilon}

\def\bG{\bf{G}}
\def\metric{\varrho}
\def\NT{\lambda}

\def\umo{\underline{\mo}}
\def\omo{\overline{\mo}}
\def\ux{\underline{x}}
\def\ox{\overline{x}}
\def\uz{\underline{z}}
\def\oz{\overline{z}}

\def\bmogbe{\boldo^\gamma_{\vep}}
\def\bxgbe{\boldx^\gamma_{\vep}}
\def\omogbe{\omo^\gamma_{\vep}}
\def\oxgbe{\ox^\gamma_{\vep}}
\def\umogbe{\umo^\gamma_{\vep}}
\def\uxgbe{\ux^\gamma_{\vep}}

\theoremstyle{plain}
\newtheorem{theorem}{Theorem}[section]
\newtheorem{lemma}[theorem]{Lemma}

\newtheorem{proposition}[theorem]{Proposition}
\newtheorem{assumption}[theorem]{Assumption}
\newtheorem{definition}[theorem]{Definition}
\newtheorem{example}[theorem]{Example}
\newtheorem{remark}[theorem]{Remark}

\DeclareMathOperator{\E}{\mathbb{E}}
\DeclareMathOperator{\F}{\mathbb{F}}

\DeclareMathOperator{\N}{\mathbb{N}}

\DeclareMathOperator{\Q}{\mathbb{Q}}
\DeclareMathOperator{\R}{\mathbb{R}}
\DeclareMathOperator{\SSS}{\mathbb{S}}

\DeclareMathOperator{\calC}{\mathcal{C}}
\DeclareMathOperator{\sD}{\mathcal{D}}

\DeclareMathOperator{\calF}{\mathcal{F}}

\DeclareMathOperator{\calO}{\mathcal{O}}

\DeclareMathOperator{\frakS}{\mathfrak{S}}

\DeclareMathOperator{\frakm}{\mathfrak{m}}

\DeclareMathOperator{\boldo}{\bf{o}}

\DeclareMathOperator{\boldx}{\bf{x}}

\DeclareMathOperator{\boldz}{\bf{z}}

\DeclareMathOperator{\boldI}{\bf{I}}

\def\sA{{\mathcal A}}

\def\sC{{\mathcal C}}

\def\cF{{\mathcal F}}

\def\cO{{\mathcal O}}

\def\sA{{\mathscr{A}}}
\def\sB{{\mathscr{B}}}
\def\sC{{\mathscr{C}}}
\def\sD{{\mathscr{D}}}
\def\sE{{\mathscr{E}}}

\def\sH{{\mathscr{H}}}
\def\sI{{\mathscr{I}}}

\def\sL{{\mathscr{L}}}
\def\sM{{\mathscr{M}}}

\def\sP{{\mathscr{P}}}
\def\sR{{\mathscr{R}}}

\def\sU{{\mathscr{U}}}

\def\sX{{\mathscr{X}}}


\numberwithin{equation}{section}
\usepackage{fancyhdr}
\date{\vspace{-1em}\normalsize{\today}}

\title{

Controlled Occupied Processes and Viscosity Solutions \\[1em]

}

\author{H. Mete Soner\footnote{Department of Operations Research and Financial
Engineering, Princeton University, Princeton, NJ, 08540,  USA. Email: 
{\tt soner@princeton.edu}. Research partially supported by the National Science Foundation grant DMS 2406762.\\[-0.75em]}
\and Valentin Tissot-Daguette\footnote{Quantitative Research, Bloomberg L.P., New York, NY, 10022, USA. Email: 
{\tt vtissotdague@bloomberg.net}\\[-0.75em]}
\and Jianfeng Zhang\footnote{Department of Mathematics,
University of Southern California, Los Angeles, CA, 90089,  USA. Email: 
{\tt jianfenz@usc.edu}. Research partially supported by the National Science Foundation grant DMS 2205972. } }

\date{\today}

\newcommand{\mc}[1]{\textcolor{purple}{#1}} 

\begin{document}
\maketitle
\vspace{5pt}

\abstract{ 
We consider the optimal control of occupied processes
which record all positions of the state process. Dynamic programming
yields nonlinear equations on the space of positive measures.
We develop the viscosity theory for this infinite dimensional
parabolic \textit{occupied} PDE by proving a comparison result
between sub and supersolutions, and thus provide a 
characterization of the value function as the unique 
viscosity solution.  Toward this proof, an extension  of the celebrated 
Crandall-Ishii-Lions (second order) Lemma to this setting,
as well as finite-dimensional approximations,
is established.
Examples including the occupied heat equation,
and pricing PDEs  of financial derivatives contingent on the occupation measure
are also discussed.

\vspace{3mm}
\textbf{Keywords:}  Stochastic optimal control, occupation  flow, occupied PDEs, viscosity solutions
\\[-0.5em]

\textbf{Mathematics Subject Classification}:  
49L12, 
35K55,  
35R15, 
60J55, 
93E20 


\section{Introduction}
This paper  introduces a class of path-dependent stochastic control problems involving the occupation 
measure and develops a viscosity theory for the associated dynamic programming equation. 
We consider the following control problem, 
\begin{align}
&\inf_{\alpha \in \sA}\E^{\Q}
[\int_{0}^{T} \ell(\calO_{t}^{\alpha},X_{t}^{\alpha}, \alpha_t) dt  
+ g(\calO_{T}^{\alpha},X_{T}^{\alpha})],    \label{eq:controlIntro} \\[1em] 
 &\d X_t^{\alpha} = b(\calO_t^{\alpha},X_t^{\alpha},\alpha_t) \d t 
  + \sigma(\calO_t^{\alpha},X_t^{\alpha},\alpha_t) \d W_t,
    \label{eq:OSDEIntro}
\end{align}
where $\calO^{\alpha}_t = \int_0^t \delta_{X_s^{\alpha}}ds$ is the 
\textit{standard time occupation flow of $X^{\alpha}$}. The latter is a measure-valued process 
that describes the cumulative time spent by $X^{\alpha}$ in arbitrary regions. 
Other clocks, possibly stochastic,  may be used in the definition of $\calO^{\alpha}$; see \cref{sec:control}.  
The pair $(\calO^{\alpha},X^{\alpha})$, referred to as the \textit{occupied process} \cite{Tissot,TissotThesis}, 
induces a Markovian framework that lies strictly  between 
 the classical (path-independent) setting, 
and  fully path-dependent models. 

Our main result characterizes the value function associated to \eqref{eq:controlIntro} 
as the unique viscosity solution to a  parabolic   \textit{occupied PDE} 
which {in standard time takes the form 

    \begin{equation}\label{eq:OPDE}
       -\pmo u + \sH(\mo,x,\nabla u,\nabla^2u) = 0.
  \end{equation}
Here, the variable $\mo$  represents the value of the occupation flow $\cO^{\alpha}$ and lives in 
the linear space $\sM$ of finite Borel measures. 
Importantly, $\mo$ 
generalizes  
the  time variable  in parabolic PDEs, and 
the \textit{occupation derivative}  $\pmo$ in \eqref{eq:OPDE} 
replaces the usual time derivative. 
As we argue in \cref{ssec:occDer}, this differential is directly tied to the dynamics of the occupation flow and corresponds to a local projection of the linear derivative commonly used in mean-field games \cite{CardDelLasLio,CarmonaDelarue,LasryLions}. 

Alternatively, one can formulate the control problem \eqref{eq:controlIntro}  using a pathwise setting and the tools from Dupire's functional Itô calculus 
\cite{DupireFITO}.
The associated dynamic programming  equation is a  
path-dependent PDE  \cite{ZhangEkren2014,EkrenTouziZhangI, EkrenTouziZhangII, PengWang, ZhangBook} which,  
in the canonical setting, reads  
\begin{equation}\label{eq:PPDE}
      -\partial_{t} u + \sH(t,\omega,\partial_{\om} u,\partial^2_{\omega}u) = 0,
\end{equation}
where $u = u(t,\omega)$,  $(t,\omega)\in \R_+ \times \ \Omega$ and
$\Omega = \calC([0,T];\R^d)$ is the space of continuous paths on  $\R^d$. 
The differential $\partial_{t}$ denotes the functional time derivative in the sense of Dupire, 
while $\partial_{\om}$ is the functional space derivative. 
Comparing equations \eqref{eq:OPDE} and \eqref{eq:PPDE}, we observe that 
occupied PDEs  correspond to  the change of coordinates %
$$
\R_+ \times \  \Omega  \ni 
        (t,\om)  \; \longrightarrow \; (\cO_t^{\alpha}(\omega),X_t^{\alpha}(\omega)) = 
        (\mo,x) \in \sM \times \R^d.
$$
In particular, the path is captured by the occupation flow, and the space variable becomes finite-dimensional. 
Consequently, the space derivatives in \eqref{eq:OPDE} are the classical ones, and the path dependence only 
appears through a first-order term given by the occupation derivative. Many questions related to  
fully nonlinear parabolic path-dependent PDEs remains open, 
especially  regarding the regularity of  solutions. Additionally, 
the comparison principle in this context is highly involved and necessitate additional 
technical tools \cite{EkrenTouziZhangI, EkrenTouziZhangII, Zhou}.  
In contrast, we here provide a  nearly classical proof
of the comparison principle for occupied PDEs, covering a large class of 
path-dependent PDEs. 
       
It is important to note that  the occupation measure 
extracts aggregate features of the path and in particular erases its chronology. 
In fact, one can show the opposite, namely that for fixed $t\ge 0$, any chronology-invariant path functional $\omega \mapsto F(t,\omega)$ 
is necessarily  function of the occupation measure \cite[Theorem 2.1]{Tissot}. 
While some path-dependent control problems cannot be 
expressed as \eqref{eq:controlIntro}, e.g those involving time delays \cite{Gollmann,SaporitoZhang}, 
the present framework still includes  a vast array of path-dependent problems, notably in finance \cite{Tissot}. 

It is now classical that the notion of viscosity solutions 
is the appropriate framework for nonlinear PDEs,
delivering a complete theory including the existence and uniqueness of such type of solutions 
\cite{userGuide,FS}.  
In our context it is particularly needed, as the regularity of solutions 
has not been established and is not expected for 
non-linear equations.
 A central tool in the analysis of 
 second-order  nonlinear PDEs is the well-known Crandall-Ishii-Lions lemma \cite{userGuide} which we generalize to our context. 
 This is achieved by projecting the problem to a finite-dimensional space 
 using \textit{cylindrical functions}, and the corresponding metric induced by 
 the so-called cylindrical norm introduced in \cref{ssec:cylindrical}. 

Classically, the comparison principle for viscosity solutions requires compactness assumptions 
on the domain so that  upper/lower semicontinuous functions attain their maximum/minimum \cite{userGuide,FS}. 
Herein, we note that  the space $\sM \times \R^d$ is not  compact with respect to the product topology. 
However, it is locally compact, and the existence of extrema is guaranteed by introducing \textit{coercive approximations}  
as  discussed in \cref{ssec:coercise}. The lack of compactness is typically reconciled with  the 
 Ekeland-Borwein-Preiss variational principle \cite{Ekeland,BorweinZhu}. 
 While 
 this powerful, yet complex tool may be applied to our context, 
 it is actually not needed here due to the locally compact structure of $\sM\times \R^d $; see again \cref{ssec:coercise}. 
\vspace{10pt}

\textbf{Related Literature.} 
In a recent work, \citet{Bethencourt} study   Brownian particles controlled by their aggregate occupation flow   
with a central example related to  the volume of their  Wiener sausage. 
This highly intriguing model has the tractable linear-quadratic structure.
Hence,  the dynamic programming equation can be solved explicitly and does not necessitate the development of a viscosity theory. 
Our setting  shares similarities with stochastic control problems in the Wasserstein space,  
including McKean-Vlasov optimal control \cite{BEZ,BL,BIRS,CGKPR,soneryan,soneryan2, ZhouTouziZhang}
and controlled superprocesses \cite{Ocello}. Additionally, the local nature of the occupation derivative in \eqref{eq:OPDE} 
is reminiscent of mean field games   with local coupling 
\cite{CardaliaguetLocal,LionsSouganidis}, where the dynamics of the system depends on specific values of the marginal densities. Moreover, \citet{BouchardTan} study regularity properties of  linear parabolic PPDEs where the state variable is enlarged with an additive functional $\omega \mapsto \int_0^t \omega_u dA_u$ 
for some continuous function $A$ of bounded variation. 
When  $A$ is non-decreasing, it can be regarded   as a random time process  (see Section \ref{sec:StochasticControl}) and their framework comes as a particular case of the present  setting. 

Infinite dimensional viscosity solutions have been discussed repeatedly in the literature and arise from different contexts,
and the classical book from Fabbri, Gozzi, and Swiech \cite{FGS} provides a comprehensive overview. 
The use of cylindrical functions as first introduced by Lasry and Lions \cite{LLR,lions_3} for nonlinear PDEs in Hilbert spaces,
that we extend to our setting.  
Relatedly, \cite{SwiechWessel} consider a finite-dimensional reduction of control problems in the Wasserstein space
using interacting particle systems, and \citet{BEZ} uses a  supremum/infimum projection  among all measures with equal barycenter. 
The cylindrical norm we introduce is also connected to the distance-like function in \cite{BIRS} pertaining to controlled McKean-Vlasov jump-diffusion processes. 
\vspace{10pt}

\textbf{Structure of the paper. } The notations are outlined in \cref{sec:notation}, where the occupation derivative and cylindrical norm are introduced. 
\cref{sec:control} expands on the stochastic control problem, the corresponding viscosity theory,  and summarizes our main results. Technical tools are presented in \cref{sec:tools}, while \cref{sec:CI} focuses on the Crandall-Ishii-Lions Lemma. The comparison principle is stated and proved in \cref{sec:comparison}. We finally discuss examples in \cref{sec:examples} and provide postponed proofs in \cref{sec:Itoproof,sec:Lipschitz}.

\section{Notations } 
\label{sec:notation}

The \emph{ambient space} is the $d$-dimensional Euclidean space $\R^d$ and $\SSS^d$ is the set of all $d$ by $d$ symmetric matrices.  
For $\Gamma \in \SSS^d$, $\text{tr}(\Gamma)$ is the trace of $\Gamma$ and $I$ is the identity matrix.
We let $\sC:=\sC(\R^d)$ be the set of continuous functions on $\R^d$ and
$\sC_b:= \sC_b(\R^d)$ be the bounded ones.
$\sM=\sM(\R^d)$ is the set of all  signed measures  on $\R^d$
endowed with the weak  $\sigma(\sC_b(\R^d),\sM)$ topology and $\sM_+$ are
the positive ones.   We  denote the total variation of $\mo  \in \sM$ by $|\mo|$ and for $\mo \in \sM_+$,  $|\mo| =\mo(\R^d)$.

For $x \in \R^d$, set $q(x):= \sqrt{1+|x|^2}$ and
$\sC_q:= \{g \in \sC\, :\, |g| \le c\, q\ \text{for some}\ c>0 \}$.
The Wasserstein space
$\sM_q:= \sM_q(\R^d)=\{ \mo \in \sM_+\, :\, \int q(x)\mo(\d x) <\infty\},$ 
is endowed with the weak  $\sigma(\sC_b(\R^d),\sM)$ topology,
despite with this topology $\sM_q$ is not complete. Given $\mo \in \sM_q,\  f \in \sC_q,$ we can define their pairing as  
$$\mo(f):= \int f(x) \mo(\d x).$$ 
Let $\sD := \sM_q \times \R^d$ be the  \emph{state space} endowed 
with the product topology. For a \emph{finite horizon} $T>0$, we set 
$$\sM_T:=\{\mo \in \sM_q : |\mo| \le T\}, \quad \sD_T:=\{ (\mo,x) \in \sD : |\mo| \le T\}, $$
with  the interior 
 $\mathring{\sD_T}:=\{ (\mo,x) \in \sD : |\mo| < T\}$, and boundary  
 $\partial \sD_T:= \{ (\mo,x) \in \sD : |\mo| = T\}$. 
The function space $\sU(\sD)$ is the set of \emph{upper semi-continuous} functions on $\sD$, 
and $\sU_b(\sD)$ are the ones that are \emph{bounded from above}.
Analogously, $\sL(\sD)$ is the set of \emph{lower semi-continuous} functions on $\sD$, 
and $\sL_b(\sD)$ are the ones that are \emph{bounded from below}.

\subsection{Occupation Derivative}
\label{ssec:occDer}
A function $\varphi :\sD \to \R$ is called \emph{locally differentiable at $(\mo,x) \in \sD$} if
 \begin{equation}\label{eq:occDerivatives}
    \pmo \varphi(\mo,x):=
\lim_{h \downarrow 0} \, \frac{\varphi(\mo + h\delta_{x},x)-\varphi(\mo,x)}{h} \quad \text{exists and is finite.}
 \end{equation}
  We say that $\varphi$ is \emph{locally differentiable} if \eqref{eq:occDerivatives} holds for all $(\mo,x) \in \sD$ and  
call $\pmo \varphi: \sD\to \R$ the \emph{occupation derivative of $\varphi$} \cite{Tissot}.  The local nature of the occupation derivative comes from the dynamics of the occupation flow given  in \cref{sec:StochasticControl}. 
We also introduce the \emph{linear derivative}  of $\varphi$ \cite{CarmonaDelarue}  and denote it by 
$\dmo \varphi(\mo,x)\in \sC(\R^d)$. We recall for convenience that $\dmo \varphi$ is characterized by the  identity
\begin{equation}
    \varphi(\mo',x)-\varphi(\mo,x)  = \int_{0}^1 \int_{\R^d}\dmo \varphi(\mo+ \eta (\mo'-\mo),x)(y)(\mo'-\mo)(dy)  d\eta, \quad \forall \mo,\mo'\in \sM.
\end{equation}
If $\pmo \varphi$ is continuous in the product topology and satisfies $|\pmo \varphi(\mo,x)|\le C(1+|x|^2)$ for some constant $C>0$, then the linear derivative of $\varphi$ exists and 
relates to the occupation derivative through 
\begin{equation}\label{eq:OccvsLinearDerivative}
    \pmo \varphi(\mo,x)=\dmo \varphi(\mo,x)(x), \quad \forall (\mo,x) \in \sD.
\end{equation}
See for instance  \cite[Theorem 2.1]{RenWang}. 
In other words, $\pmo \varphi: \sD\to \R$ is  a local projection of the linear derivative. 
Finally, $\varphi \in \sC(\sD)$ is said to be in $\sC^{1,2}(\sD)$ if it has 
derivatives $\pmo \varphi$, $\nabla \varphi$, $\nabla^2 \varphi$ that are continuous.

\subsection{Cylindrical Norm}
\label{ssec:cylindrical}

Let $(f_k)_{k\in \N} \subset  \sC^{1}_b(\R^d)$ 
be a separating family
of $\sM$.  Namely, $\mo(f_k) = \mo'(f_k)$ for all $k\in \N$ 
implies that $\mo= \mo'$; see, e.g., \cite[Section 2.1]{Ocello}. 
Suppose that  $f_{0}$ is a scalar multiple of} $\tilde{f}_0 \equiv 1$  and by normalizing, we assume that
\begin{equation}
\label{eq:summable}
\sum_{k} \lVert f_k \rVert_{\sC^1}^2  \le 1,\qquad
\text{where}\qquad
\lVert f \rVert_{\sC^1}:= \lVert f \rVert_{\infty} + \lVert \nabla f \rVert_{\infty} .
\end{equation}
Using this family,  we introduce  the \textit{cylindrical}
and the \emph{parabolic} norms, given respectively   by
$$
\varrho(\mo) := \big( \sum\nolimits_{k} |\mo(f_k)\, |^2\big)^{\tfrac12}, \qquad 
\rho( \mo, x) := \sqrt{\metric^2(\mo)+|x|^2},\qquad
\mo \in \sM, \ x \in \R^d.
$$
 In particular, $\varrho(\mo) = 0$ if and only if $\mo=0$, since  $(f_k)_{k\in \N}$ 
is a separating family
of $\sM$. 
As $|\mo(f_k)| \le \lVert f_k\rVert_{\infty}|\mo|\le  \lVert f_k \rVert_{\sC^1} |\mo|$,
we conclude that $\metric(\mo) \le |\mo|$. 
Hence, the metric induced by  $\metric$  is weaker than the total variation distance. In fact, it is easily seen that  on $\sM_T$, $\varrho$ metrizes the weak topology.  
Moreover, for $\mo \in \sM_+$, we directly calculate that  
\be \label{eq:DerRho}
\pmo \rho^2(o,x) = \dmo \metric^2(\mo)(x)
= 2\, \sum_k \mo(f_k) f_k(x).
\ee

\section{Stochastic Control Problem}\label{sec:StochasticControl}

Consider a filtered probability space  $(\Omega,\calF,\F = (\calF_t)_{t\in [0,T]},\Q)$
satisfying the usual conditions. Let  $X$ be a 
continuous semi-martingale,
and $\Lambda_t\ge 0$ be a \emph{strictly
increasing} continuous adapted process.
Then, the \textit{occupation flow of $X$ with random time $\Lambda$} is
defined by,
\begin{equation}
\label{eq:occFlow}
\cO = (\calO_t)_{t\ge 0}, \quad 
\calO _t(B) :=
\int_0^t\mathds{1}_B (X_s) 
\d \Lambda_s,  \quad B\in \sB(\R^d).
\end{equation} 
Clearly for each $t >0$, the support of $\cO _t$ is included in 
the ball of radius $\sup_{s \in [0,t]} |X_s|$ centered at the origin.
Therefore, $\cO _t \in \sM_q$ and $(\cO _t,X_t) \in \sD$.
We call the pair $(\calO ,X)$ an \textit{occupied process}.
For all the properties of these processes
we refer the reader to  \cite{Tissot}, and the references therein.

Two important examples of $\Lambda$ are the followings:
\begin{itemize}
\item \emph{Standard time occupation flow} corresponds to the
choice $\Lambda_t=t$.
\item Let
$\Lambda_t= \text{tr}(\langle X \rangle_t)$, where $\langle X \rangle = (\langle X^i,X^j \rangle)_{i,j}$
is the covariation matrix of $X$.  Then,
we call $\cO$  the \emph{occupation flow of $X$},
which generates the local times \cite{RevuzYor,Tissot}.
\end{itemize}

\subsection{Controlled Occupied Processes}\label{sec:control}

Let $A$ be a Borel subset of a Euclidean space and 
the set $\sA$ of \emph{admissible controls} are all predictable
process $\alpha = (\alpha_t)_{t\ge 0}$ taking values in $A$  and 
$\Q \times {\text{Lebesgue}}$ square integrable.
For fixed $\alpha \in \sA$, we let  $(\calO^{\alpha},X^{\alpha})$  be the  
 strong solution of the 
 controlled \textit{occupied stochastic differential equation} (OSDE), 
 \begin{align}
  \d \calO_t^{\alpha} &= \NT(\calO_t^{\alpha},X_t^{\alpha},\alpha_t)
  \  \delta_{X_t^{\alpha}} \d t, 
  \label{eq:OSDEOcc}
  \\  \d X_t^{\alpha} &= b(\calO_t^{\alpha},X_t^{\alpha},\alpha_t) \d t 
  + \sigma(\calO_t^{\alpha},X_t^{\alpha},\alpha_t) \d W_t,
    \label{eq:OSDE}
\end{align}
with given functions $\NT:\sD\times A \to \R_+$, $b:\sD\times A \to \R^d$, $\sigma:\sD\times A \to \R^{d\times d}$ and initial condition $(\calO_0^{\alpha},X_0^{\alpha}) = (\mo,x)\in \sD$.  
First equation implies that $\calO^{\alpha}$ 
is the  occupation flow with random time
$\Lambda_t^{\alpha}= |\mo| + \int_0^t\NT(\calO_s^{\alpha},X_s^{\alpha},\alpha_s)\d s$.
If we let $\NT=\lVert \sigma \rVert_{F}^2 = \text{tr}(\sigma \sigma^{\! \top})$
be the squared Frobenius norm of $\sigma$, then $\cO^\alpha$
is the occupation flow of the process $X^\alpha$.
The choice $\NT \equiv 1$ corresponds to the 
standard time occupation flow of $X^\alpha$.

Following \cite{Tissot}, we make the following definition.
\begin{definition}
\label{def:general}
{\rm{For a normed vector space $(E,|\cdot|)$, we say a function $\varphi: \sD_T \times A \mapsto E$
satisfies}} the growth and Lipschitz conditions with constant $c_*$, 
{\rm{if the followings hold
for every $(\mo,x,a), (\mo',x',a)$ in the set $\sD_T \times A$:}}
\begin{align}
\label{eq:growth}
&|\varphi(\mo,x,a)|  \le c_* (1+\rho(\mo,x)),\\ 
\label{eq:Lipschitz}
&|\varphi(\mo,x,a)-\varphi(\mo',x',a)| \le c_*\rho(\mo-\mo',x-x'). 
\end{align} 
\end{definition}

\begin{lemma}
\label{lem:exists}
Suppose that the coefficients $\NT, b, \sigma$
satisfy the growth and Lipschitz conditions with a constant $c_*$ with the Euclidean and Frobenius norm for  $\NT, b$ and  $\sigma$, respectively. 
Then, there is a unique strong solution of OSDE
\eqref{eq:OSDEOcc}$-$\eqref{eq:OSDE} for any
initial condition $(\mo,x) \in \sD_T$.
\end{lemma}
\begin{proof}
This result is proved in \cite[Theorem 4.2.11]{TissotThesis}
under the growth and Lipschitz conditions  
with  
$\rho_{\text{\tiny BL}}(\mo,x) := \sqrt{|\mo|_{\text{\tiny BL}} +|x|^2}$, where $|\cdot|_{\text{\tiny BL}}$ is  
the \textit{bounded Lipschitz norm} 
$$|\mo|_{\text{\tiny BL}} := \sup\left\{\mo(f) : \max(\lVert f \rVert_{\infty}, [f]_{\text{Lip}})\le 1\right\}, \quad [f]_{\text{Lip}} := \sup_{x\ne x'}\frac{|f(x) - f(x')|}{|x-x'|}.$$ 
We now verify that $\rho_{\text{\tiny BL}}$ dominates $\rho$. 
Indeed, define $\tilde{f}_k = f_k/\lVert f_k \rVert_{\sC^1}\in \sC^{1}_b(\R^d)$  with $(f_k)$ given in \cref{ssec:cylindrical}. Then $\max(\lVert \tilde{f}_k \rVert_{\infty}, [\tilde{f}_k]_{\text{Lip}}) \le \lVert \tilde{f}_k\lVert_{\sC^1} \le  1$, hence
$$\varrho(\mo)^2 = \sum\nolimits_{k} |\mo(f_k)\, |^2 \le \sum\nolimits_{k} \lVert f_k \rVert_{\sC^1}^2 |\mo(\tilde{f}_k)\, |^2\le |o|_{\textnormal{\tiny{BL}}}^2\sum\nolimits_{k} \lVert f_k \rVert_{\sC^1}^2 \le |o|_{\textnormal{\tiny{BL}}}^2.$$
 Then it is clear that $\rho^2(\mo,x) \le \rho_{\text{\tiny BL}}^2(\mo,x)$.
Thus, the conditions of 
\cite[Theorem 4.2.11]{TissotThesis} hold and the existence of
a strong solution follows.
\end{proof}

The following It\^o formula is a generalization of Theorem 3.3 in \cite{Tissot}. See also \cite[Proposition 9]{Bethencourt}. 
\begin{proposition}
    \label{prop:Ito} \textnormal{\textbf{(Itô Formula)}}
  Consider an occupied process  $(\cO,X)$  with dynamics 
    $$d\calO_t = \lambda_t \delta_{X_t}dt, \quad dX_t =b_tdt + \sigma_t dW_t,$$ where  $\lambda,b,\sigma$ are given adapted processes taking values in $\R_+,\R^d, \R^{d\times d}$, respectively,  such that $\lambda$ and $b$ are locally integrable and $\sigma$ is locally square integrable.  
    If $v\in \sC^{1,2}(\sD)$, then  for all $t\ge 0$,
       \begin{align}
       dv(\calO_t,X_t)  &=  \Big(\NT_t  \pmo v + b_t \cdot \nabla v+  \frac{1}{2}\textnormal{tr}(\sigma_t \sigma^{\! \top}_t \nabla^2 v) \Big)(\calO_t,X_t) dt + \nabla  v(\calO_t,X_t)\cdot \sigma_tdW_t.\label{eq:ito}
    \end{align} 
\end{proposition}
\begin{proof}
    See \cref{sec:Itoproof}.
\end{proof}
Evidently, the above formula applies to controlled occupied processes 
solving \eqref{eq:OSDEOcc}-\eqref{eq:OSDE} by setting $\varphi_t = \varphi(\calO_t^{\alpha},X_t^{\alpha},\alpha_t)$, $\varphi\in \{\lambda,b,\sigma\}$. 

\subsection{Value function}
\label{ssec:value}
 
Given $\ell :\sD_T \times A \to \R$, 
$g:\partial\sD_T\to \R$, and the shorthand notation $\E^{\Q}_{\mo,x}
[\cdot] = \E^{\Q}
[\ \cdot \ |  \ (\calO^{\alpha}_0,X^{\alpha}_0) = (\mo,x)]$,  consider the   control problem of minimizing
\begin{equation}
\label{eq:controProblem}
J(\mo,x,\alpha) := \E^{\Q}_{\mo,x}
[\int_{0}^{\tau^{\alpha}} \ell(\calO_{t}^{\alpha},X_{t}^{\alpha}, \alpha_t) dt  
+ g(\calO_{\tau^{\alpha}}^{\alpha},X_{\tau^{\alpha}}^{\alpha})],
\end{equation}
over all $\alpha \in \sA$,
with the exit time 
$\tau^{\alpha} = \inf\{t\ge 0 : (\calO_t^{\alpha},X_t^{\alpha})\notin 
\mathring{\sD_T}\} = \inf\{t\ge 0 : \Lambda_t^{\alpha} \ge T\}.$ 
The \emph{value function} is given by,
\be \label{eq:valueFct}
v(\mo,x):= \inf_{\alpha \in \sA} J(\mo,x,\alpha),\qquad (\mo,x) \in \sD_T.
\ee
In what follows, we make the 
following standing assumption.
\begin{assumption}
\label{asm:standing}
There exists $c_* \ge 1$ 
such that $\NT$, $b$, $\sigma$, $\ell$, and $g$ all
satisfy growth and Lipschitz conditions with $c_*$. Moreover,
\be
\label{eq:WeakEllipticity}
\NT(\mo,x,a) \ge 1/c_*,\qquad \forall \ (\mo,x,a) \in \sD_T \times A.
\ee
\end{assumption}
We remark that we require the nondegeneracy of $\lambda$, but not that of $\sigma$. In particular, 
condition \eqref{eq:WeakEllipticity} gives a lower bound on the total mass process, namely  
$$
\Lambda_t^{\alpha} = |\cO_0^\alpha|  +\int_0^t 
\NT(\calO_{s}^{\alpha},X_{s}^{\alpha}, \alpha_s) \, d s
\ge t/ c_*,
\qquad t \ge0.
$$
Hence the exit time in \eqref{eq:controProblem} satisfies $\tau^{\alpha} \le c_{*} T$, ensuring that 
the objective function $J$  is finite. If $\lambda \equiv 1 $ (standard time occupation flow), then \eqref{eq:WeakEllipticity} holds with $c_*=1$ and we have   $\tau^{\alpha} \equiv T$. If $\NT=\lVert \sigma \rVert_{F}^2$ (occupation flow), then \eqref{eq:WeakEllipticity} can be seen as a weak ellipticity condition on $\sigma$. 

\begin{proposition} \label{prop:Lipschitz}
      Under Assumption \ref{asm:standing}, the value function is locally $1/2-$Hölder  
   continuous with respect to $\rho$, that is, for all $\delta >0$ 
   there exists  $\hat{c} >0$ that depends on $c_*$, $T$, and $\delta$ such that 
   \begin{equation}
      \rho(\mo-\mo',x-x') \le \delta \; \Longrightarrow \; |v(\mo,x) - v(\mo',x')| \le \hat{c} \ \rho(\mo-\mo',x-x')^{1/2}. 
   \end{equation}
   
\end{proposition}
\begin{proof}
    See \cref{sec:Lipschitz}.
\end{proof}
 
\begin{remark}
    \label{rem-Holder}
    In standard stochastic control theory, typically the value function is $1/2-$Hölder continuous in $t$, but is Lipschitz continuous in $x$. Here, however,  $v$ is only $1/2-$Hölder continuous in $x$  in general. The reason is that the exit time $\tau^\alpha$ may depend on $x$. For example, let $d=1$, $b\equiv 0$, $\sigma \equiv 1$, $\ell \equiv 0$, and $\lambda = \lambda(x)\ge 1$,  $g=g(x)$ are Lipschitz continuous in $x$ with Lipschitz constant $C=1$. Then the problem does not involve the control $\alpha$. Given $(\mo, x) \in \sD_T$, the exit time $\tau_{\mo,x} = \inf\{t\ge 0: \int_0^t \lambda(x+W_s) ds \ge  T-|\mo|\}$, and $v(\mo, x) = \E^{\Q}
[g(x+W_{\tau_{\mo,x}})]$. Thus
\begin{align*}
    |v(\mo,x)-v(\mo,x')| 
\le |x-x'| +  \E^{\Q}
\big[|W_{\tau_{\mo,x}}-W_{\tau_{\mo,x'}}|\big]
\end{align*}
From the Burkholder-Davis-Gundy inequality \cite[Chapter IV]{RevuzYor}, we can see that the last term above is of order $\E^{\Q}
\big[|\tau_{\mo,x}-\tau_{\mo,x'}|^{1/2}\big] \sim |x-x'|^{1/2}$. 
\end{remark}
   
 \subsection{Viscosity Solutions and Uniqueness}
 \label{sec:main}
For $(\mo,x) \in \sD_T$, $\zeta:=(\theta,\Delta,\Gamma) \in \R\times \R^d \times \SSS^d$,
the Hamiltonian is given by, 
\begin{equation}\label{eq:Hamiltonian}
 \sH(\mo,x,\zeta) = -\inf_{a \in A} \Big(\NT(\mo,x,a) \ \theta 
     +   b(\mo,x,a) \cdot \Delta + \frac{1}{2}\textnormal{tr}( (\sigma \sigma^{\! \top})(\mo,x,a) \Gamma)  
     + \ell (\mo,x,a) \Big).
\end{equation} 
 Then, by \cite{FS}  the \textit{dynamic programming equation}  
 associated to \eqref{eq:controProblem}  is given by 
\begin{subnumcases}{}
\sH(\mo,x,\pmo u,\nabla u,\nabla^2u) = 0 & \text{ on }   
$\mathring{\sD_T}$, \label{eq:HJB12}\\
      u = g, & \text{ on }  $\partial  \sD_{T} $, 
      \label{eq:HJB22} 
  \end{subnumcases}
  where $\pmo u = \pmo u(\mo,x)$ is the occupation derivative. As announced in the Introduction, equation  \eqref{eq:HJB12} is 
  linear in $\pmo u$ when $\lambda \equiv 1 $ (standard time) and admits the more familiar expression, 
  \begin{align*}
       -\pmo u + \sH_x(\mo,x,\nabla u,\nabla^2u) = 0, \quad  \sH_x(\mo,x,\Delta,\Gamma) = \sH(\mo,x,0,\Delta,\Gamma).
  \end{align*}
  In general, however, 
  the control may influence the rate $\lambda$ 
  leading to nonlinearities in the occupation derivative  as well. 

Next, we introduce suitable families of test functions towards a viscosity theory.
For $T>0$, $u \in \sU(\sD_T)$, $w \in \sL(\sD_T)$,
and $(\mo,x)\in \sD_T$, following \cite{userGuide, FS} we set 
     \begin{align*}
      \frakS_{+}u(\mo,x) = \frakS_{+}^Tu(\mo,x)  &:= \{\phi \in \sC^{1,2}(\sD_T)\, : \,
       0=(u-\phi)(\mo,x)  = \max_{\sD_T} (u-\phi) \},
            \\
        \frakS_{-}w(\mo,x) =\frakS_{-}^Tw(\mo,x)  &:= \{\phi \in \sC^{1,2}(\sD_T)\, : \,
         0=(w-\phi)(\mo,x)  = \min_{\sD_T} (w-\phi) \}.
        \end{align*}
The definition of viscosity solutions is classical \cite{userGuide, FS}.
        
\begin{definition}\label{def:visc}
        \textnormal{\textbf{(Viscosity solutions)}} We say that
\begin{itemize}        
\item {\rm{$w\in \sL(\sD_T)$ is a}}
           viscosity supersolution {\rm{of $\sH \ge 0$, if}} 
\be
\label{eq:viscSub}
         \sH(\mo,x,\pmo \phi,\nabla \phi,\nabla^2\phi)  \ge 0,  
         \qquad \forall \phi \in \frakS_{-}w(\mo,x), \ \ (\mo,x) \in
         \mathring{\sD_T}. 
\ee
\item  {\rm{$u\in \sU(\sD_T)$ is a}} viscosity subsolution 
    {\rm{of $\sH \le 0$, if}}
\be
\label{eq:viscSuper}
     \sH(\mo,x,\pmo \phi,\nabla \phi,\nabla^2\phi)  \le 0,  
     \qquad \forall \phi \in \frakS_{+}u(\mo,x), \ \ (\mo,x) \in 
     \mathring{\sD_T}.    
\ee
\item {\rm{$v\in \sC(\sD_T)$ a}} viscosity solution $\sH=0$,
     {\rm{if it is both a viscosity supersolution of $\sH \ge 0$ and a subsolution of $\sH\le 0$.}}
\end{itemize}
\end{definition}

\begin{remark}
        \label{rem:semi}
        {\rm{Following the classical theory,
         in the definition of a supersolution we can take
        any function (not even measurable) and then consider its 
        upper semicontinuous envelope.  Similarly for 
        subsolutions, we could consider the lower semicontinuous envelope.}} 
 \end{remark}

The main result of this paper is the following comparison principle.
 \begin{theorem}
\label{thm:compare}  \textnormal{\textbf{(Comparison)}} Suppose that Assumption \ref{asm:standing} holds,
$\uv \in \sU_b(\sD_T)$ is a viscosity subsolution of $\sH \le 0$,
$\ov\in \sL_b(\sD_T)$  is a viscosity  supersolution of $\sH \ge  0$,
and $\uv\le \ov $  on $\partial \sD_T$. Then,  $u \le w$ on $\sD_T$. 
\end{theorem}
We shall devote the next three sections to the proof of this theorem. We now provide the characterization of the value function $v$. 

\begin{theorem}
\label{thm:main}
Suppose that Assumption \ref{asm:standing} holds. Then, the value function $v$ is the unique viscosity solution of the dynamic programming equation  \eqref{eq:HJB12}--\eqref{eq:HJB22}.
\end{theorem}
\begin{proof} 
Given the regularity in Proposition \ref{prop:Lipschitz} and using the techniques developed in \cite{FS}, it is classical that the value function $v$ is a viscosity solution. 
Moreover, the uniqueness of viscosity solution is a direct consequence of Theorem \ref{thm:compare}.
\end{proof}

\section{Technical tools}
\label{sec:tools}
In this section, we outline several concepts that we utilize.

\subsection{Coercivity and coercive  approximations}
\label{ssec:coercise}

We use the notion of coercivity to construct
extrema of functions
on a topological space $\sX$.
\begin{definition}
\label{def:coercive}
{\rm{We say that a function $ u \in \sU(\sX)$ is}}
upper coercive {\rm{if all hypographs of $u$
are sequentially compact, i.e., for every $c \in \R$, every
sequence $\{\xi_n\}_n$ in $\sX$ with $u(\xi_n) \ge c$
has a limit point $\xi^* \in \sX$.
$\sU_e(\sX)$ denotes the set of upper coercive functions.}} 

{\rm{Similarly, we say that a function 
$w \in \sL(\sX)$ is}}
lower coercive {\rm{if every epigraph of $w$
is sequentially  compact, and $\sL_e(\sX)$ denotes the set of lower coercive functions.}} 
\end{definition}

A direct consequence of the definition is 
the following.
\begin{lemma}
\label{lem:coercive}
Any $u \in \sU_e(\sX)$  is bounded from above
and achieves its maximum on any closed subset of $\sX$.
Analogously, any $w \in \sL_e(\sX)$
is bounded from below and achieves its minimum 
 on any closed subset of  $\sX$.
\end{lemma}
\begin{proof}
Consider an upper coercive function $u \in \sU(\sX)$.
Towards a contraposition, suppose that for each 
positive integer $n$, there is $\xi_n \in \sX$ such that
$u(\xi_n) \ge n$. Since $\{ u \ge 1\}$ is compact by hypothesis,
on a subsequence $n_k$, $\xi_{n_k}$ is convergent.
Let $\xi_*$ be the limit point.  Since $u$ is upper
semicontinuous and real-valued, $\infty > u(\xi_*) 
\ge \lim_k u(\xi_{n_k}) = \infty$.  Hence, $u$ is bounded from above.
Then, the upper semicontinuity
and the upper coercivity implies that the
maximum of $u$ on any closed set is achieved.
\end{proof}
Recall from \cref{sec:notation} that $q(x)= \sqrt{1+|x|^2\, }$ and 
set 
\be
\label{eq:q}
\vt(\mo,x):= \mo(q)+ q(x), \qquad (\mo,x) \in \sD=\sM_q \times \R^d.
\ee
It is classical that the function $\vt \in \sL_b(\sD)$, and 
for any constant $c>0$,
the sub-level set  $\{ (\mo,x) \in \sD \, :\,
\vt(\mo,x) \le c\ \}$ is compact.
Then,
$$
\sD = \cup_{m \ge 1} \sD^m, \qquad
\sD^m := \{(\mo,x) \in \sD \, :\, \vt(\mo,x) \le m\}.
$$
Although $\sD$ is not compact,
each $\sD^m$ is compact giving it a locally
compact structure.  Additionally, $\vt \in \sC^{1,2}(\sD)$ 
and the derivatives are given by,
\be
\label{eq:vtder}
\pmo \vt(\mo,x) = q(x),\qquad
\nabla \vt(\mo,x)= \frac{x}{q(x)},\qquad
\nabla^2 \vt(\mo,x)= \frac{1}{q(x)}\Big(I - \frac{x \otimes x}{q(x)^2}\Big).
\ee

For  $w\in \sL(\sD)$,  $u \in \sU(\sD)$,  and $\gamma >0$, we 
define \emph{coercive approximations} by,
\be
\label{eq:dcoercive}
w_\gamma:= w +\gamma \vt,\qquad
u^\gamma:= u- \gamma \vt.
\ee

\begin{lemma}
\label{lem:delta_coercive} 
For any $w \in \sL_b(\sD)$,  $w_\gamma \in \sL_e(\sD)$,
and for any $u \in \sU_b(\sD)$,  $u^\gamma \in \sU_e(\sD)$.
\end{lemma}
\begin{proof}
As $\vt \in \sL_b(\sD)$, it is clear that  $w_\gamma \in \sL_b(\sD)$.
Additionally,
as $w$ is bounded from below and
lower semicontinuous, for any constant $c$,
$\{w_\gamma \le c\}$ is a closed subset of
$\sD^{m_*}$ with ${m_*}=(c - \inf _\sD w)/\gamma$. 
Since $\sD^m$ is compact for every $m$, 
any closed subset of it is also compact.  Hence,
$w_\gamma$ is lower coercive.  The proof for $u$ is essentially
identical.
\end{proof}
       
\subsection{Semijets}\label{sec:semijets}
Following definition is classical in the theory of 
viscosity solutions \cite{userGuide,FS}.
\begin{definition}
{\rm{For $(\mo,x) \in \sD_T$, the}} parabolic superjet 
 {\rm{of $u \in \sU(\sD_T)$ at $(\mo,x)$ is given by,}}
 $$
\sP_{+}^{1,2}u(\mo,x):= \{ ( \pmo \phi(\mo,x), 
\nabla \phi(\mo,x), \nabla^2 \phi(\mo,x))\, :\,
\phi \in  \frakS_{+}u(\mo,x)\}.
$$
 {\rm{The}} parabolic  subjet 
 {\rm{of $w \in \sL(\sD_T)$ at $(\mo,x)\in \sD_T$ is given by,}}
 $$
\sP_{-}^{1,2}w(\mo,x):= \{ ( \pmo \phi(\mo,x), \nabla \phi(\mo,x), \nabla^2 \phi(\mo,x))\, :\,
\phi \in  \frakS_{-}w(\mo,x)\}.
$$
\end{definition} 
Above sets are subsets of $\R \times \R^d \times \SSS^d$,
and their closures are defined by,
$$
\overline{\sP}_{\pm}^{1,2}v(\mo,x) := \{\lim_{n\to \infty}\zeta_n \, : \,  
\zeta_n \in \sP_{\pm}^{1,2}v(\mo_n,x_n), \ 
(\mo_n,x_n, v(\mo_n,x_n)) \rightarrow (\mo,x, v(\mo,x))  \}.
$$

The following equivalent definition of
viscosity solutions follows directly from the
continuity of the Hamiltonian and the definitions.
Indeed, $u \in \sU(\sD_T)$ satisfies \eqref{eq:viscSuper} if and only if
 \begin{equation}
 \label{eq:viscSupJet}
     \sH(\mo,x,\zeta) \ge 0, \qquad \forall \ \zeta \in \overline{\sP}_{-}^{1,2}u(\mo,x),\
     (\mo,x) \in \mathring{\sD_T}.
 \end{equation}
 Similarly, $w\in \sL(\sD_T)$ satisfies \eqref{eq:viscSub} if and only if
 \begin{equation}
 \label{eq:viscSubJet}
     \sH(\mo,x,\zeta) \le 0, \qquad \forall \ \zeta \in \overline{\sP}_{+}^{1,2}w(\mo,x),\
     (\mo,x) \in \mathring{\sD_T}.
 \end{equation}
 The following is a direct consequence of the definitions.
 \begin{lemma}
 \label{lem:Jcont}
 Suppose that $\zeta_n \in \overline{\sP}^{1,2}_{\pm}v(\mo_n,x_n)$, and
$(\zeta_n, \mo_n,x_n, v(\mo_n,x_n)) \to (\zeta, \mo,x, v(\mo,x))$
as $n$ tends to infinity.
 Then,  $\zeta \in \overline{\sP}^{1,2}_{\pm}v(\mo,x)$.
 \end{lemma}
 
\subsection{Finite-dimensional Projections}
\label{ssec:finite}

We use a regularization technique similar
to the one introduced by Lasry and Lions \cite{LLR} and 
related finite-dimensional projections by Lions \cite{lions_3}
to prove the uniqueness of viscosity solutions 
in Hilbert spaces. 
Recall the separating class $(f_k)_{k\in \N} \subset  \sC^{1}_b(\R^d)$  of subsection 
\ref{ssec:cylindrical}.
For a positive integer $K$,
we  define the finite-dimensional projections by
$$
 \pi_K(\mo) := (\mo(f_1),\ldots,\mo(f_K) )
\in \sR_K := \pi_K(\sM_T).
$$
The range $\sR_K$ depends  on $T$ and is a compact subset of $\R^K$. 
Indeed, as $|\pi_K(\mo)| \le \varrho(\mo) \le |\mo|\le T$,  $\sR_K$ is contained in $\{z\in \R^K  :   |z| \le T\}$.
All the projected objects we define in this subsection, like $\sR_K$,
may depend on $T$, but we suppress this dependence.

We project a given $v :\sD_T \to \R$
onto the \emph{projected state space} $\sR_K\times \R^d$ 
as follows,
\begin{align*}
\Pi^K(v)(z,x) &:= \sup \{v(\mo,x)\, :\, \pi_K(\mo)=z\}, \\
\nonumber
\Pi_K(v)(z,x) &:= \inf \{v(\mo,x)\, :\, \pi_K(\mo)=z\},\qquad
(z,x) \in \sR_K \times \R^d.
\end{align*}
Let $\sM^K(v,z,x)$ be the maximizers of the first expression,
and $\sM_K(v,z,x)$ be the minimizers of the second expression.
Although these sets might be empty, in view of Lemma \ref{lem:coercive},
$\sM^K(u,z,x)$ is non-empty when $u \in \sU_e(\sD)$,
and  $\sM_K(w,z,x)$ is non-empty if  $w \in \sL_e(\sD)$.
We make the simple yet crucial observation that 
\begin{equation}
\label{eq:squeezeV}
    \Pi_K(v)(\pi_K(\mo),x) \le v(\mo,x) 
    \le \Pi^K(v)(\pi_K(\mo),x), \quad \forall (\mo,x) \in \sD.
\end{equation}

\begin{lemma}
\label{lem:Kcoercive}
For any $K$, $u \in \sU_e(\sD_T)$, 
and $w \in \sL_e(\sD_T)$,
we have $\Pi^K(u) \in \sU_e(\sR_K \times \R^d)$
and $\Pi_K(w) \in \sL_e(\sR_K \times \R^d)$.
\end{lemma}

\begin{proof} Set $w_K:=\Pi_K(w)$
and consider a sequence $(z_n,x_n)$ with 
$w_K(z_n,x_n) \le c$ for some constant $c$.
Choose $\mo_n \in \sM_K(w,z_n,x_n)$.
Then, $w(\mo_n,x_n) =w_K(z_n,x_n) \le c$,
and by the coercivity of $w$, there is $(\mo_*,x_*) \in \sD$
and a subsequence,
denoted by $n$ again, such that
$\lim_n (\mo_n,x_n) =(\mo_*,x_*)$.
Then, $\lim_nz_n = \lim_n \pi_K(\mo_n)
= \pi_K(\mo_*) =:z_*$ which follows from the weak convergence of $\mo_n$ to $\mo_*$. 
Hence, $(z_*,x_*)$ is a
limit point of the sequence $\{(z_n,x_n)\}_n$,
proving the compactness of epigraphs of $w_K$
in the space $\sR_K \times \R^d$.

Consider a sequence $\{(z_n,x_n)\}_n$
satisfying   $\lim_n (z_n,x_n,w_K(z_n,x_n))=(z_*,x_*,w_*)$. 
To prove the lower semicontinuity of
$w_K$, we need to show that $w_* \ge  w_K(z_*,x_*)$.
Indeed, by the same argument as above,
any sequence $\mo_n \in \sM_K(w,z_n,x_n)$
has a limit point $\mo_*$.  In light of $\pi_K(\mo_*)=z_*$
and \eqref{eq:squeezeV}, $w_K(z_*,x_*) \le w(\mo_*,x_*)$.
We now use the lower
semicontinuity of $w$ to arrive at
$$
w_*=\lim_n w_K(z_n,x_n) =\liminf_n w(\mo_n,x_n) \ge 
w(\mo_*,x_*) \ge w_K(z_*,x_*).
$$
\end{proof}

The sub and super differentials and the sub and superjets 
in finite dimensional spaces are classical.  Indeed, 
for any  $\tilde v :\sR_K\times \R^d \to \R$, we follow
\cite{userGuide,CI}  and define 
the parabolic  sub and superjets 
$\sP_{\pm}^{1,2}\tilde{v}(z,x), \overline{\sP}_{\pm}^{1,2}\tilde{v}(z,x)$
as subsets of $\R^K \times \R^d \times \SSS_d$
using the Taylor expansion restricted to the set $\sR^K \times \R^d$.  
These sets depend on the domain $\sR_K\times \R^d $
of $\tilde v$ and they have the following representation (see
\cite[Page 11]{userGuide}),
\begin{align*}
\sP_{+}^{1,2}\tilde v(z,x)&= \{ ( \nabla_z \phi(z,x), 
\nabla_x \phi(z,x), \nabla_x^2 \phi(z,x))\, :\,
\phi \in  \frakS_{+}\tilde v(z,x)\}\\
\sP_{-}^{1,2}\tilde v(z,x)&= \{ ( \nabla_z \phi(z,x), 
\nabla_x \phi(z,x), \nabla_x^2 \phi(z,x))\, :\,
\phi \in  \frakS_{-}\tilde v(z,x)\},
\end{align*}
where for $(z,x) \in \sR_K \times \R^d$,
 \begin{align*}
 \frakS_{+}\tilde v(z,x) &:= \{\phi \in \sC^{1,2}(\R^K\times \R^d)\, : \,
0=(\tilde v-\phi)(z,x)  = \max_{\sR_K \times  \R^d} (\tilde v-\phi) \},\\
\frakS_{-}\tilde v(z,x) &:= \{\phi \in \sC^{1,2}(\R^K\times \R^d)\, : \,
0=(\tilde v-\phi)(z,x)  = \min_{\sR_K \times  \R^d} (\tilde v-\phi) \}.
\end{align*}
We remark that here we do not require $(z,x)$ to be an interior point of $\sR_K \times \R^d$, see Remark \ref{rem:warning} below. 
We map the sub and superjets of $\tilde v$ into smaller sets as follows:
$$
\overline{\sP}_{K,\pm}^{1,2}\tilde{v}(z,x) 
:=\{ (\tilde\theta\cdot(f_1(x),\ldots,f_K(x)),\Delta,\Gamma)\, :\,
(\tilde\theta,\Delta,\Gamma) \in \overline{\sP}_{\pm}^{1,2}\tilde{v}(z,x) \},
$$
and $\sP_{K,\pm}^{1,2}\tilde{v}(z,x)$ is defined analogously.
Then,
for any $v:\sD_T \to \R$, $\tilde v :\sR_K\times \R^d \to \R$,
$$
\sP_{\pm}^{1,2}v(\mo,x),\  \overline{\sP}_{\pm}^{1,2}v(\mo,x), \
\sP_{K,\pm}^{1,2}\tilde{v}(z,x),\ \overline{\sP}_{K,\pm}^{1,2}\tilde{v}(z,x) 
\subset \R \times \R^d \times \SSS_d.
$$
Using the 
projection maps $\Pi_K, \Pi^K$,
we connect the semijets in $\sR_K\times \R^d$ with those in $\sD_T$. 
\begin{lemma}
\label{lem:maximizers}
For any $(z,x)\in \sR_K\times \R^d$ and
$w \in \sL(\sD_T)$,
$$
\sP_{K,-}^{1,2}w_K(z,x) \subset \sP_-^{1,2}w(\mo,x), \quad \forall  \mo \in \sM_K(w,z,x),
$$
where $w_K:=\Pi_K(w)$.  If additionally $w$ is lower coercive,
then
\begin{equation*}
\overline{\sP}_{K,-}^{1,2}w_K(z,x) \subset \bigcup_{\mo \in \sM_K(w,z,x)} \overline{\sP}_{-}^{1,2}w(\mo,x).
\end{equation*}

Similarly, for any $(z,x)\in \sR_K\times \R^d$ and 
$u \in \sU(\sD_T)$ with  $u^K:=\Pi^K(u)$,
$$
\sP_{K,+}^{1,2}u^K(z,x) \subset \sP_+^{1,2}u(\mo,x), \quad \forall  \mo \in \sM^K(u,z,x).
$$
If $u \in \sU_e(\sD)$, then
\begin{equation*}
\overline{\sP}_{K,+}^{1,2}u^K(z,x) \subset \bigcup_{\mo \in \sM^K(u,z,x)} \overline{\sP}_+^{1,2}u(\mo,x).
\end{equation*}
\end{lemma}

\begin{proof}
Fix $(z,x) \in \sR_K \times \R^d$ and $(\theta,\Delta,\Gamma) \in \sP_{K,-}^{1,2}w_K(z,x)$.  Then
by definition, there exists
$\varphi \in \frakS_{+}w_K(z,x)$ such that
$$
\theta = \nabla_z \varphi(z,x)\cdot (f_1(x),\ldots,f_K(x)),\
\Delta = \nabla_x \varphi(z,x),\
\Gamma = \nabla_x^2 \varphi(z,x).
$$
Since $\varphi \in \frakS_{+}w_K(z,x)$, we have $ \varphi \le w_K$.
Set also 
$$
\phi(\mo,x):= \varphi(\pi_K(\mo),x), \quad (\mo,x) \in \sD.
$$
Then, for any $(\mo',x')\in \sD$, using \eqref{eq:squeezeV} and 
$ \varphi \le w_K$ we arrive at
$$
w(\mo',x') \ge w_K(\pi_K(\mo'),x') \ge \varphi(\pi_K(\mo'),x') = \phi(\mo',x').
$$

Additionally for any $\mo \in \sM_K(w,z,x)$, $\pi_K(\mo)=z$
and 
$$
w(\mo,x) = w_K(z,x) = \varphi(z,x) = \phi(\mo,x).
$$
Summarizing, we have shown that,
$$
0= (w-\phi)(\mo,x) \le (w-\phi)(\mo',x'),\qquad 
\forall  (\mo',x') \in \sD_T.
$$
Hence, $\phi \in \frakS_{+}w(\mo,x)$.  Moreover,
\begin{align*}
\dmo \phi(\mo,x)(\cdot)&= 
\nabla_z \varphi(\pi_K(\mo),x) \cdot (f_1(\cdot),\ldots, f_K(\cdot)),\\[0.5em]
\Longrightarrow \quad
\pmo \phi(\mo,x)&=\dmo \phi(\mo,x)(x)=
\nabla_z \varphi(\pi_K(\mo),x) \cdot (f_1(x),\ldots, f_K(x))=\theta.
\end{align*}
Since $\nabla_x \phi(\mo,x)= \nabla_x \varphi(z,x)=\Delta$ and
$\nabla_x^2 \phi(\mo,x)= \nabla_x^2 \varphi(z,x)=\Gamma$, we conclude
that 
$$
(\theta,\Delta,\Gamma) \in \sP_{-}^{1,2}w(\mo,x).
$$

Now suppose that  $\zeta = (\theta,\Delta,\Gamma) \in \overline{\sP}_{K,-}^{1,2}w_K(z,x)$.
Then, there are $(x_n,z_n) \to (x,z)$ and $\zeta_n \to \zeta$ such that 
$\zeta_n \in \sP_{K,-}^{1,2}w_K(z_n,x_n)$ and 
$w_K(z_n,x_n) \to w_K(z,x)$.  Since $w \in \sL_e(\sD)$,
$\sM_K(w,z_n,x_n)$ is non-empty and by the above result there is
$\mo_n \in \sM_K(w,z_n,x_n)$ such that $\zeta_n \in \sP_{+}^{1,2}w(\mo_n,x_n)$.
In particular, since $w_K(z_n,x_n) \to w_K(z,x)$, the sequence $w_K(z_n,x_n)$ is uniformly
bounded from above.  Set $m:=\sup_n w_K(z_n,x_n)$. 
As $w(\mo_n,x_n) = w_K(z_n,x_n)$, 
$(\mo_n, x_n) \in \{ w \le m\}$.
Then by the lower coercivity
of $w$, there is $(\mo,x) \in \sD$ and a subsequence, denoted by $n$ again, such that
$(\mo_n,x_n) \to (\mo,x)$.  

We claim that $\mo \in \sM_K(w,z,x)$
and $\zeta \in \overline{\sP}_-^{1,2}w(\mo,x)$.  Indeed,
$$
\pi_K(\mo)= \lim_n  \pi_K(\mo_n)= \lim_n z_n = z,
\quad \Longrightarrow \quad
w(\mo,x) \ge w_K(\pi_K(\mo),x)=w_K(z,x).
$$
Moreover, since $\lim_n w_K(z_n,x_n) = w_K(z,x)$
and $w\in \sL(\sD_T)$, 
$$
w_K(z,x) \le w(\mo,x)
\le \liminf_n w(\mo_n,x_n)
=\lim_n w_K(z_n,x_n) = w_K(z,x).
$$
Hence, $\lim_n w(\mo_n,x_n)=w(\mo,x)=w_K(z,x)$,
implying that $\mo \in \sM_K(w,z,x)$.
Since the sequence $(\zeta_n,\mo_n,x_n, w(\mo_n,x_n))$
converges to $(\zeta,\mo,x,w(\mo,x))$ and
$\zeta_n \in \sP_{K,-}^{1,2}w_K(z_n,x_n)$,
we conclude that $\zeta \in \overline{\sP}_{-}^{1,2}w(\mo,x)$, 
proving the second statement for $w$.

The proof for $u$ is essentially identical.
\end{proof}

\section{Crandall-Ishii-Lions (Second-Order)  Lemma}
\label{sec:CI}

We write $\boldx = (\ux,\ox)$, $\boldo = (\umo,\omo)$,
$\pi_K(\boldo) = (\pi_K(\umo), \pi_K(\omo))$, 
 and 
for a given $\Phi:\sD_T^2 \to \R$, we also  write $\Phi(\boldo,\boldx) = \Phi(\umo,\ux,\omo,\ox)$ . 
Let $I$ be the identity matrix in $\SSS^d$. 
For $\Gamma_1,\Gamma_2 \in \SSS^{d}$, introduce the block diagonal matrices
in $\SSS^{2d}$ by, 
\begin{equation*}
    \text{diag}(\Gamma_1,\Gamma_2) = 
    \begin{pmatrix}
        \Gamma_1 & 0 \\
        0 & \Gamma_2
    \end{pmatrix},
    \qquad
   \bG =   \begin{pmatrix}
      I & -I \\
        -I & I
 \end{pmatrix} .
    \end{equation*}
Then,  ${\bG \boldx \cdot \boldx} = |\ux-\ox|^2$,
and $\boldI = \text{diag}(I,I) $ is the identity matrix in $\SSS^{2d}$.

\begin{lemma}\label{lem:CI}
   \textnormal{\textbf{(Crandall-Ishii-Lions Lemma)}}  Suppose that $\uv \in \sU_e(\sD_T)$, $\ov \in \sL_e(\sD_T)$,
    $\varepsilon >0$, and  $(\boldo^*,\boldx^*) = (\umo^*,\omo^*,\ux^*,\ox^*)$ 
    is a maximizer of 
     \begin{equation*}
     \label{eq:DOUBLE}
     \Phi: \sD_T^2 \to \R, \quad   \Phi(\boldo,\boldx) = 
     \uv(\umo,\ux)-\ov(\omo,\ox) - \frac{1}{2\varepsilon}\rho(\umo-\omo,\ux-\ox)^2.  
     \end{equation*}
     Then, there exist
     $\underline{\Gamma},\overline{\Gamma} \in \SSS^d$ such that 
     \begin{align}
         (\Theta(\ux^*),\Delta,\underline{\Gamma}) \in 
         \overline{\sP}_+^{1,2}\uv(\umo^*,\ux^*), 
         \qquad  (\Theta(\ox^*&),\Delta,\overline{\Gamma}) \in 
         \overline{\sP}_-^{1,2}\ov(\omo^*,\ox^*), 
         \label{eq:CI1} 
         \\[0.5em]
         -\frac{3}{\varepsilon}
          \boldI
           \le 
\textnormal{diag}( \underline{\Gamma}, -\overline{\Gamma})
    \le &\frac{3}{\varepsilon}\bG,
    \label{eq:CI2}
     \end{align}
     where $\Delta = (\ux-\ox)/\varepsilon$, and $\Theta \in \sC_b(\R^d)$ is given by 
\be
     \label{eq:Theta}
    \Theta (x) = \dmo \Psi(\umo^*)(x) = \frac{1}{\varepsilon}\sum_{k} 
   (\umo^*-\omo^*)(f_k)  f_k(x), \quad
   \Psi(\mo):= \frac{1}{2\varepsilon} \varrho^2(\mo-\omo^*).
\ee 
\end{lemma}

\begin{proof} We proceed in several steps.
We first assume that $(\boldo^*,\boldx^*)$ is a strict maximizer, and 
consider the general case at the final step.  
\vspace{4pt}

\noindent
{\emph{Step 1 (Projection)}}.   Set
$\uv^K:= \Pi^K(\uv)$,   $\ov_K:= \Pi_K(\ov)$,
and 
$$
\Phi_K(\boldz,\boldx)  :=   
\uv^K(\uz,\ux)-\ov_K(\oz,\ox) 
- \frac{1}{2\varepsilon}\left(|\uz-\oz|^2+|\ux-\ox|^2 \right).
$$
By Lemma \ref{lem:Kcoercive}, $\uv^K \in \sU_e(\sR_K \times \R^d)$
and $\ov_K \in \sL_e(\sR_K \times \R^d)$.  
Therefore,  $\Phi_K$ achieves its maximum at, 
say, $(\boldz_K,\boldx_K)$.
Additionally, by \eqref{eq:squeezeV},
 \be
 \label{eq:max}
 \Phi(\boldo,\boldx) \le\Phi_K(\pi_K(\boldo),\boldx) \le  
 \Phi_K(\boldz_K,\boldx_K),
 \qquad \forall (\boldo, \boldx) \in \sD^2_T.
 \ee
Moreover,  for any $\boldo \in \sE_K(\boldz,\boldx):= \{ \boldo=(\umo,\omo)  \in \sM_T^2\, :\,
\umo \in \sM^K(u,\uz,\ux), \, \omo \in \sM_K(w,\oz,\ox)\, \}$,
\begin{align*}
0\le \Phi_K(\boldz,\boldx)  - \Phi(\boldo,\boldx)   
&= \frac{1}{2\varepsilon}(\varrho^2(\umo-\omo) -  |\pi_K(\umo) -\pi_K(\omo)|^2)\\[0.5em]
&=\frac{1}{2\varepsilon}\,  \sum_{k >K} |(\umo-\omo)(f_k)|^2 
\le \frac{1}{2\varepsilon}\, \sum_{k >K} \|f_k\|_{\infty}^2 \, |\umo-\omo|^2 .
\end{align*}
Since $|\mo|\le T$ for all $\mo \in \sM_T$ and 
$\sum_k \|f_k\|_\infty^2 \le 1$,
we conclude that
\be
\label{eq:uniform}
\lim_{K \to \infty} \sup
\{ |\Phi(\boldo,\boldx) -\Phi_K(\boldz,\boldx)|\, :\,
\boldo \in \sE_K(\boldz,\boldx), (\boldz,\boldx) \in \sR_K \times \R^d\}= 0.
\ee  

\noindent
{\emph{Step 2 (Finite dimensional Crandall-Ishii-Lions)}}. 
By  the classical Crandall-Ishii-Lions Lemma on the closed set $\sR_K\times \R^d$
\cite[Theorem 3.2]{userGuide}, \cite{CI}, there exists   $\underline{\Gamma}_K,\overline{\Gamma}_K \in \SSS^d$ such that  
      \begin{align}
(\underline{\theta}_K,\Delta,\underline{\Gamma}_K) &\in \overline{\sP}_{+}^{1,2}\uv^K(\uz_K,\ux_K), 
\quad  (\overline{\theta}_K,\Delta,\overline{\Gamma}_K) \in \overline{\sP}_{-}^{1,2}\ov_K(\oz_K,\ox_K), 
\label{eq:CI1Pf} \\[1em]
 -\frac{3}{\varepsilon}
           \boldI
           &\le 
\textnormal{diag}( \underline{\Gamma}_K, -\overline{\Gamma}_K)
    \le \frac{3}{\varepsilon} \bG,
 \label{eq:CI2Pf}
     \end{align}
     where
     $\Delta = (\ux_K - \ox_K)/\varepsilon$, and
     $$
     \underline{\theta}_K = \frac{1}{\varepsilon}(\uz_K - \oz_K) 
     \cdot (f_1(\ux_K),\ldots, f_K(\ux_K)),\quad
     \overline{\theta}_K = \frac{1}{\varepsilon}(\uz_K - \oz_K) 
     \cdot (f_1(\ox_K),\ldots, f_K(\ox_K)).
     $$  
     Moreover, by \eqref{eq:CI1Pf} and \cref{lem:maximizers}, there exist 
     $\boldo_K=(\umo_K,\omo_K) \in \sE_K(\boldz_K,\boldx_K)$. such that 
\be
\label{eq:JK}
(\underline{\theta}_K ,\Delta,\underline{\Gamma}_K) \in \overline{\sP}_+^{1,2}\uv(\umo_K,\ux_K), 
\qquad  (\overline{\theta}_K ,\Delta,\overline{\Gamma}_K) \in \overline{\sP}_-^{1,2}\ov(\omo_K,\ox_K).
\ee

\noindent
{\emph{Step 3 (Convergence of $(\boldo_K,\boldx_K)$)}}.  We claim that 
$\lim_K (\boldo_K,\boldx_K) = (\boldo^*,\boldx^*)$.  Indeed,
by coercivity, arguing as before we conclude that on a subsequence, denoted by $K$ again,
$\lim_K(\boldo_K, \boldx_K)=:(\hat \boldo, \hat \boldx)$ exists.    
Since $\boldo_K \in \sE_K(\boldz_K,\boldx_K)$,
by \eqref{eq:uniform},
$$
\lim_{K \to \infty}  |\Phi(\boldo_K,\boldx_K) -\Phi_K(\boldz_K,\boldx_K)| =0.
$$
Moreover by \eqref{eq:max}, $\Phi_K(\boldz_K,\boldx_K)
 \ge \Phi(\boldo^*,\boldx^*)$.
Combining all these inequalities, we arrive at
$$
\Phi(\boldo^*,\boldx^*)
\ge \Phi(\hat \boldo,\hat \boldx)
\ge \limsup_K \Phi(\boldo_K,\boldx_K)
= \limsup_K\Phi_K(\boldz_K,\boldx_K) \ge
\Phi(\boldo^*,\boldx^*),
$$
where in the second inequality, we used the upper
semi-continuity of $\Phi$.
As $(\boldo^*,\boldx^*)$ is the strict maximizer of $\Phi$,
we conclude that $(\hat \boldo, \hat \boldx)=(\boldo^*,\boldx^*)$
and all above inequalities are equalities. In particular,
\be
\label{eq:uv_converge}
\lim_{K \to \infty} \uv(\umo_K,x_K)=   \uv(\umo^*,x^*),
  \quad \text{and}\quad
   \lim_{K \to \infty} \ov(\omo_K,x_K) =
   \ov(\omo^*,x^*).
\ee

\noindent
{\emph{Step 4 (Passage to limit)}}.
As $\boldz_K= \pi_K(\boldo_K)$, 
we have
$$
     \underline{\theta}_K = \frac{1}{\varepsilon} \, \sum_{k \le K}\, 
      (\umo^*-\omo^*)(f_k) f_k(\ux_K),\qquad
	\overline{\theta}_K = \frac{1}{\varepsilon}\,  \sum_{k \le K}\,
     (\umo^*-\omo^*)(f_k) f_k(\ox_K),
$$
  and recalling $\Theta$ of \eqref{eq:Theta}, this implies that
  $
  \lim_{K \to \infty} (\underline{\theta}_K,\overline{\theta}_K) 
  =  (\Theta(\ux^*),\Theta(\ox^*)).
  $
  Additionally, \eqref{eq:CI2Pf} implies that the sequences 
  $(\underline{\Gamma}_K)$, $(\overline{\Gamma}_K)$ lie in compact subsets of $\SSS^d$. 
  Hence, there exist 
  a subsequence, denoted by $K$ again,
 and $(\underline{\Gamma}, \overline{\Gamma})$
  satisfying the inequalities  \eqref{eq:CI2}, such that
    $$
  \lim_{K \to \infty} (\underline{\Gamma}_K,\overline{\Gamma}_K) 
  = (\underline{\Gamma}, \overline{\Gamma}).
  $$
 \vspace{4pt}
 In view of \eqref{eq:JK},  \eqref{eq:uv_converge}
  and Lemma \ref{lem:Jcont},  we conclude that 
  $$
 (\Theta(\ux^*),\Delta,\underline{\Gamma}) \in \overline{\sP}_+^{1,2}\uv(\umo^*,\ux^*), 
 \qquad  
 (\Theta(\ox^*),\Delta,\overline{\Gamma}) \in \overline{\sP}_-^{1,2}\ov(\omo^*,\ox^*).
$$

\noindent
\emph{Step 5 (Final step).}
Suppose that $(\boldo^*,\boldx^*)
 = (\umo^*,\omo^*,\ux^*,\ox^*)$ is a maximizer of $\Phi$ which is not necessarily strict.  
 For $ (\umo,\ux),(\omo,\ox) \in \sD_T$, set
 $$
 \tilde{\uv}(\umo,\ux):= \uv(\umo,\ux) -  \rho(\umo-\umo^*,\ux-\ux^*)^4,\qquad
 \tilde{\ov}(\omo,\ox):= \ov(\omo,\ox) +  \rho(\omo-\omo^*,\ox-\ox^*)^4.
$$
 Then, $(\boldo^*,\boldx^*) $     is a \emph{strict} maximizer of 
$$
 \tilde{ \Phi}: \sD_T^2 \to \R, \quad   \bar{\Phi}(\boldo,\boldx) = 
\tilde{\uv}(\umo,\ux)-\tilde{\ov}(\omo,\ox) - \frac{1}{2\varepsilon}\rho(\umo-\omo,\ux-\ox)^2.
$$
Moreover, 
$$
\overline{\sP}_+^{1,2}\uv(\umo^*,\ux^*)
= \overline{\sP}_+^{1,2}\tilde{\uv}(\umo^*,\ux^*),
\quad \text{and} \quad
 \overline{\sP}_-^{1,2}\ov(\omo^*,\ox^*)
= \overline{\sP}_-^{1,2}\tilde{\ov}(\omo^*,\ox^*).
$$
As $\tilde{\uv} \in \sU_e(\sD_T)$, $\tilde{\ov} \in \sL_e(\sD_T)$,
we can apply the  above steps to $\tilde{\uv}, \tilde{\ov}$, 
constructing elements in their sub and super-differentials
with desired properties.
 \end{proof}
 
 \begin{remark}
 \label{rem:warning}
We emphasize that, in the proof of the Crandall-Ishii-Lions Lemma we work on the  closed set $\sR_K \times \R^d$, including its boundary points. In particular, in \eqref{eq:CI1Pf} we do not require $(\uz_K, \ux_K)$ and $(\oz_K, \ox_K)$ to be interior points of $\sR_K \times \R^d$. In fact, this is also the case in the classical paper \cite{CI} and Theorem 3.2 of \cite{userGuide} which is stated on a general subset $\Om$ of a Euclidean space. When we apply this lemma to prove the comparison principle in the next section, however, the viscosity property holds only for interior points of the infinite-dimensional set $\sD_T$. 
 \end{remark}

 \section{Comparison Principle: Proof of Theorem \ref{thm:compare}}
\label{sec:comparison}

Recall the norms $\varrho,\rho$ defined in subsection \ref{ssec:cylindrical}
and $q$ of \eqref{eq:q}.  
For $(\gamma_1,\gamma_2,\beta) \in \R^3$, 
$(\mo,x) \in \sD_T$, $(\theta,\Delta,\Gamma)
\in \R \times \R^d \times \SSS^d$, set
$$
\sH^\gamma_\beta (\mo,x,\theta,\Delta,\Gamma)  
:= \sH (\mo,x,\theta-\beta +\gamma_1 q(x),\Delta + \gamma_2 \nabla q(x),
\Gamma+ \gamma_2 \nabla^2  q(x)).
$$
We next state the main property 
of the Hamiltonian needed in the comparison result.
In the below definition, $(\umo,\ux), (\omo,\ox)\in \sD_T$,
$\underline{\theta}, \overline{\theta} \in \R$
are arbitrary points.

\begin{definition}
\label{def:hamiltonian}
{\rm{We say that $\sH : \sD_T \times \R \times \R^d \times \SSS^d \mapsto \R$,
has the}} Crandall-Ishii-Lions (CIL) property, {\rm{if there exists
a continuous strictly
increasing modulus  $\frakm: [0,\infty) \mapsto  [0,\infty)$
with  $\frakm(0)=0$, and constants $C_*,c_0>0$,
such that for all $\varepsilon, \gamma_1,\gamma_2,\beta \in (0,1]$ we have,
\begin{align*}
\sH^{-\gamma}_{-\beta}(\omo,\ox,\overline{\theta},
\Delta_\vep,\overline{\Gamma}_\vep)
&-\sH^\gamma_\beta(\umo,\ux,\underline{\theta},
\Delta_\vep,\underline{\Gamma}_\vep)  \\[0.5em]
&\le - c_0 \beta+ \frakm(\zeta_\vep)
+C_* \big[ \frac{1}{\vep} \rho^2(\umo-\omo,\ux-\ox)
+ \gamma_1 Q^2(\ux,\ox)+\gamma_2 Q(\ux,\ox)\big],
\end{align*}
where $\Delta_{\varepsilon}:= (\ux-\ox)/\varepsilon$, 
$Q(\ux,\ox):=q(\ux)+q(\ox)$,
$\underline{\Gamma}_\vep,\overline{\Gamma}_\vep
\in \SSS^d$  is any pair satisfying \eqref{eq:CI2}, and
$$
\zeta_\vep:= Q(\ux,\ox)[\varrho(\umo-\omo)+\rho(\umo-\omo,\ux-\ox) (|\Delta_\vep|
+|\underline{\theta}|+|\overline{\theta}|)
+|\underline{\theta}-\overline{\theta}|].
$$}}
\end{definition}

We first establish this property under natural conditions. 
\begin{lemma} 
\label{lem:CI-verify} 
Under Assumption \ref{asm:standing}, the Hamiltonian $\sH$ defined in \eqref{eq:Hamiltonian} 
has the Crandall-Ishii-Lions property.
\end{lemma}

\begin{proof}
First we note that for any $(\mo,x) \in \sD_T$,
\be\label{eq:rho}
\rho(\mo,x)=\sqrt{\varrho(\mo)^2+|x|^2} \le |\mo|+|x|
\le T +|x| \le (T+1) q(x).
\ee
Also, as $q\ge 1$, $q(x)\nabla q(x)=x$, and $ q(x)\nabla^2 q(x) = I - x \otimes x/q^2(x)$,
we have,
$$
|\nabla q(\ux)|\le 1,\quad
\|\nabla^2 q(\ux)\|_F \le d.
$$
For $\vep>0$ set
$$
\Xi_\vep:=(\umo,\ux,\underline{\theta}, \Delta_\vep,\underline{\Gamma}_\vep),
\qquad
\Xi^\vep:=(\omo,\ox,\overline{\theta}, \Delta_\vep,\overline{\Gamma}_\vep).
$$
By  ellipticity \eqref{eq:WeakEllipticity}, for any $\beta >0$,
$$
\sH^\gamma_\beta(\Xi_\vep)
\ge \sH^\gamma_0(\Xi_\vep) +\frac{1}{c_*} \beta,\qquad
\sH^{-\gamma}_{-\beta}(\Xi^\vep)
\le \sH^{-\gamma}_0(\Xi^\vep) -\frac{1}{c_*} \beta.
$$
Set $c_0:=2/c_*$,
$\underline{G}(\gamma, \Xi_\vep):= \sH(\Xi_\vep)- \sH^\gamma_0(\Xi_\vep)$, and
$\overline{G}(\gamma,\Xi^\vep):=  \sH^{-\gamma}_0(\Xi^\vep)- \sH(\Xi^\vep)$,
so that
\begin{align*}
\sH^{-\gamma}_{-\beta}(\Xi^\vep)-\sH^\gamma_\beta(\Xi_\vep)
&\le -c_0 \beta + \sH^{-\gamma}_0(\Xi^\vep)-\sH^\gamma_0(\Xi_\vep) \\[0.5em]
&=-c_0\beta +\sH(\Xi^\vep)-\sH(\Xi_\vep)
+\overline{G}(\gamma, \Xi^\vep)+\underline{G}(\gamma, \Xi_\vep) .
\end{align*}
We continue by estimating $\underline{G}(\gamma, \Xi_\vep)$ 
and $\overline{G}(\gamma, \Xi^\vep)$ using the uniform Lipschitz
continuity assumption on $\sigma, b, \ell$.
In  view of Assumption \ref{asm:standing},
\begin{align*}
    |\underline{G}(\gamma, \Xi_\vep)| \le & 
    \sup_{a \in A} \big\{\,  \gamma_1  \NT(\umo,\ux,a)\, q(\ux)
    +   \gamma_2\, \big [
    |b(\umo,\ux,a)|\, |\nabla q(\ux)| + \frac12 |\text{tr}((\sigma \sigma^{\text{T}})(\umo,\ux,a) 
    \nabla^2 q(\ux))|\big ]\big \}\\[0.5em]
    &\le  c_*  \rho(\umo, \ux)
(\gamma_1 q(\ux) +\gamma_2[ |\nabla q(\ux)| +  \|\nabla^2 q(\ux)\|_F]) \\[0.5em]
    &\le  \gamma c_* (1+d) (1+T)(\gamma_1 q^2(\ux) +\gamma_2 q(\ux)).
\end{align*}
We similarly show that
$$
|\overline{G}(\gamma, \Xi^\vep)|  
 \le  \gamma c_* (1+d) (1+T)(\gamma_1 q^2(\ox) +\gamma_2 q(\ox)).
$$
Set $C_1:=  c_* (1+d) (1+T)$, to conclude that
$$
\sH^\gamma_\beta(\Xi_\vep)-\sH^{-\gamma}_{-\beta}(\Xi^\vep)
\le -c_0 \beta +\sH(\Xi_\vep) - \sH(\Xi^\vep)
+\gamma C_1 (\gamma_1Q^2(\ux,\ox) +\gamma_2Q(\ux,\ox)).
$$

We next estimate  $\sH(\Xi_\vep) - \sH(\Xi^\vep)$
in several steps. Set
\begin{align*}
\sI_1&:= \sup_{a \in A} \{ |\NT(\umo,\ux,a) 
-\NT(\omo,\ox,a)|\, |\underline \theta|\}
+\sup_{a \in A} \{ \lVert \sigma(\omo,\ox,a) \rVert_{F}^2\, 
|\underline \theta-\overline \theta| \}\\[0.5em]
&\le c_* \rho(\umo-\omo,\ux-\ox)  |\underline \theta|
+c_* (1+\rho(\omo,\ox))  |\underline \theta-\overline \theta|,\\[1em]
\sI_2&:= \sup_{a \in A} \{ |b(\umo,\ux,a) -b(\omo,\ox,a)| |\Delta_\vep|\}
\le c_* \rho(\umo-\omo,\ux-\ox)  |\Delta_\vep|.
\end{align*}
In view of \eqref{eq:rho}, $\rho(\umo-\omo,\ux-\ox) \le (T+1)(q(\ux)+q(\ox)) = (T+1)Q(\ux,\ox)$. Hence, 
$$
\sI_1+\sI_2 \le c_*(T+1) Q(\ux,\ox)|\underline \theta-\overline \theta|
+  c_* \rho(\umo-\omo,\ux-\ox)(|\Delta_\vep|+|\underline \theta| +\overline \theta|).
$$

Finally,  set
\begin{align*}
    \sI_3&:= \sup_{a \in A}\{  \,
    {\text{tr}}((\sigma \sigma^{\! \top}
)(\umo,\ux,a) \underline\Gamma_\vep- 
    (\sigma \sigma^{\! \top})(\omo,\ox,a) \underline\Gamma^\vep)\, \}.
\end{align*}
We estimate $\sI_3$ as in 
\cite[Example 3.6]{userGuide} using  \eqref{eq:CI2}. Then,
for every $a \in A$,
\begin{align*}
&{\text{tr}}\left((\sigma \sigma^{\! \top})(\umo,\ux,a) \underline\Gamma_\vep - 
    (\sigma \sigma^{\! \top})(\omo,\ox,a) \underline\Gamma^\vep)\right)\\[0.5em]
    &\hspace{50pt}=
{\text{tr}}\left( {\text{diag}}(\underline \Gamma_\vep, -\overline \Gamma^\vep)\,
\begin{pmatrix}
    \sigma (\umo,\ux,a)\\
    \sigma (\omo,\ox,a)\,
\end{pmatrix}
\left(\sigma^{\! \top}(\umo,\ux,a),\sigma^{\! \top}(\omo,\ox,a)\right) \right)\\[0.5em]
&\hspace{50pt}\le \frac{3}{\vep}
{\text{tr}}\left( {\mathbf{G}}\,
\begin{pmatrix}
    \sigma (\umo,\ux,a)\\
    \sigma (\omo,\ox,a)\,
\end{pmatrix}
\left(\sigma^{\! \top}(\umo,\ux,a),\sigma^{\! \top}(\omo,\ox,a)\right) \right)\\[0.5em]
&\hspace{50pt}=\  \frac{3}{\vep} \|\sigma (\umo,\ux,a)- \sigma (\omo,\ox,a)\|_F^2
\le \frac{3 c_*^2}{\vep}\rho^2(\umo-\omo,\ux-\ox).
\end{align*}
We now use all above inequalities to arrive at
\begin{align*}
    |\sH(\Xi_\vep)& - \sH(\Xi^\vep)|  \le
    \sI_1 +\sI_2+\sI_3\\[0.5em]
   &\le  c_*(T+1) Q(\ux,\ox)|\underline \theta-\overline \theta|
+2  c_* (|\Delta_\vep|+|\underline \theta| +\overline \theta|)+
  \frac{3 c_*^2}{\vep}\rho^2(\umo-\omo,\ux-\ox).
    \end{align*}
\end{proof}


We are now ready to prove the main result.

\noindent{\bf Proof of Theorem \ref{thm:compare}}.  Towards a counterposition we assume that
$\sup_{\sD_T} (u -w) >0$ and
proceed in several steps to obtain a contradiction.
\vspace{4pt}

\noindent
{\emph{Step 1  (Set up).}}  
We first note that the occupation derivative of 
$\psi(\mo,x) := |\mo|$  is equal to  $\pmo \psi(\mo,x)=  1$ for
any $(\mo,x) \in \sD$.
Next, for $(\gamma_1,\gamma_2,\beta) \in (0,1]$,
$(\umo,\ux), (\omo,\ox) \in \sD_T^2$ we set,
\begin{align*}
\uvgb(\umo,\ux)&:= u(\umo,\ux) -\gamma_1 \umo(q) -\gamma_2q(\ux)
+\beta(|\umo|-T),\\[0.5em]
\ovgb(\omo,\ox)&:= w(\omo,\ox) +\gamma_1 \omo(q) +\gamma_2q(\ox)
-\beta(|\omo|-T).
\end{align*}
A direct argument using the fact that
$\pmo \psi(\mo,x)=  1$ 
shows that $\uvgb$ is a viscosity subsolution 
of $\sH^\gamma_\beta \le 0$ in $\mathring{\sD_T}$, and $\ovgb$ is a viscosity supersolution of $\sH^{-\gamma}_{-\beta} \le 0$.  Moreover by Lemma \ref{lem:coercive},
 $\uvgb\in \sU_e(\sD_T)$, $\ovgb\in \sL_e(\sD_T)$. Hence, there is $\beta_0>0$ such that
\be
\label{eq:counter}
\sup_{\sD_T} (\uvgb -\ovgb) >0 \ge
\sup_{\partial \sD_T} (\uvgb -\ovgb),
\qquad \forall \, 0 \le \beta,\gamma_1, \gamma_2 \le \beta_0 .
\ee
In the remainder of the proof, we fix $\beta=\beta_0>0$ and assume that
the parameters always satisfy $0<\gamma_1,\gamma_2 \le \beta_0$.
Additionally, to simplify the presentation we write
$$
\uvg:=\uv^\gamma_{\beta_0},
\quad \text{and}\quad
\ovg:=\ov_\gamma^{\beta_0}.
$$

\noindent
{\emph{Step 2  (Doubling the variables).}} 
For $(\boldo,\boldx)=((\umo,\ux), (\omo,\ox)) \in \sD_T^2$, 
$\gamma_1,\gamma_2, \vep \in (0,\beta_0]$, set
$$
\Phi(\boldo,\boldx) := \Phi^\gamma_{\vep}(\boldo,\boldx)=
\uvg(\umo,\ux)-\ovg(\omo,\ox) - \frac{1}{2\vep} \rho^2(\umo-\omo,\ux-\ox). 
$$
By Lemma \ref{lem:coercive},  there exists  a maximizer 
$(\bmogbe,\bxgbe)=((\umogbe,\uxgbe),(\omogbe,\oxgbe))
\in \sD_T^2$ of $\Phi^\gamma_{\vep,\beta}$. 
Then, in view of \cref{lem:CI},  there exists 
$(\underline{\Gamma}_{\vep},\overline{\Gamma}_\vep)=
(\underline{\Gamma}^\gamma_\vep,\overline{\Gamma}_\vep^\gamma)
\in (\SSS^d)^2$ satisfying
\eqref{eq:CI2} such that
\begin{align}
&\qquad\quad   (\underline{\theta}^\gamma_\vep,
\Delta_\vep,\underline{\Gamma}_\vep)
\in \overline{\sP}_+^{1,2}\uvg(\umogbe,\uxgbe), 
\qquad  
( \overline{\theta}^\gamma_\vep,
\Delta_\vep,\overline{\Gamma}_\vep) 
\in \overline{\sP}_-^{1,2}\ovg(\omogbe,\oxgbe),
\label{eq:CI1EpsGamma} \\[1em]
\nonumber
&\underline{\theta}^\gamma_\vep := \Theta^\gamma_\vep(\uxgbe),\qquad
\overline{\theta}^\gamma_\vep: = \Theta^\gamma_\vep(\oxgbe), \quad \Theta^\gamma_\vep = \frac{1}{\vep}\sum\nolimits_{k} (\umogbe -\omogbe)(f_k) \, f_k,\quad
\Delta_\vep :=\Delta^\gamma_\vep = \frac{1}{\vep}(\uxgbe-\oxgbe).
\end{align}

\noindent
 {\emph{Step 3 (Viscosity property).}}
By \eqref{eq:counter}, for all sufficiently small $\gamma,\vep$, 
maximizers of $\Phi^\gamma_\vep$
satisfy $((\umogbe,\uxgbe),(\omogbe,\oxgbe)) \in \mathring{\sD_T}^{\!\!2}$. Then,
 in view of  \eqref{eq:CI1EpsGamma} and the viscosity properties of 
 $\uvg$, $\ovg$, 
 $$
\sH^\gamma_{\beta_0}(\umogbe,\uxgbe,\underline{\theta}^\gamma_\vep, 
 \Delta_\vep,\underline{\Gamma}_\vep)\, \le\, 0 \, \le\,
 \sH^{-\gamma}_{-\beta_0}(\omogbe,\oxgbe,\overline{\theta}^\gamma_\vep, 
 \Delta_\vep,\overline{\Gamma}_\vep).
 $$
We now use the above inequalities and 
   the CIL property from Lemma \ref{lem:CI-verify} to arrive at,
\begin{align}
 \label{eq:H}
0  &\le  \sH^{-\gamma}_{-\beta_0}(\omogbe,\oxgbe,\overline{\theta}^\gamma_\vep, 
\Delta_\vep,\overline{\Gamma}_\vep) -
\sH^\gamma_{\beta_0}(\umogbe,\uxgbe,\underline{\theta}^\gamma_\vep, 
\Delta_\vep,\underline{\Gamma}_\vep)\\[0.5em]
\nonumber
 &\le 
 - c_0 \beta_0+ \frakm(\zeta_\vep^\gamma)
+ C_* [ \rho^2(\umogbe-\omogbe,\uxgbe-\oxgbe)/\vep + \gamma_1 Q^2(\uxgbe,\oxgbe)
+\gamma_2 Q(\uxgbe,\oxgbe)],
 \end{align}
 where  as before $Q(\uxgbe,\oxgbe)=q(\uxgbe)+q(\oxgbe)$ and
 $$
 \zeta_\vep^\gamma=
 Q(\uxgbe,\oxgbe)\, [\varrho(\umogbe-\omogbe)+
 \rho(\umogbe-\omogbe,\uxgbe-\oxgbe)\, 
 (|\Delta_\vep| +|\underline{\theta}^\gamma_\vep|+|\overline{\theta}^\gamma_\vep|)\, 
 +  |\underline{\theta}^\gamma_\vep-\overline{\theta}^\gamma_\vep|].
 $$
 
\noindent
 {\emph{Step 4 (Passage to limit).}} 
 We let the parameters tend to zero in the inequality \eqref{eq:H}
 to obtain a contradiction resulting from \eqref{eq:counter}, hence
 completing the proof.  In order, we first let $\vep \downarrow 0$ 
 then $\gamma_1 \downarrow 0$, and finally $\gamma_2 \downarrow 0$.
 
 We start by using the standard direct arguments as in \cite[Lemma 3.1]{userGuide} to obtain 
\be
\label{eq:limzero}
 \lim_{\vep \downarrow 0}
 \frac{1}{\vep}\, \rho^2(\umogbe -\omogbe,\uxgbe-\oxgbe) =0,\qquad
 \lim_{\gamma_2 \downarrow 0} \,
 \lim_{\gamma_1 \downarrow 0}\, \lim_{\vep \downarrow 0}\,
\gamma_2\,  Q(\uxgbe,\oxgbe)=0.
\ee
In particular, 
\be
\label{eq:Qge}
c(\gamma_1):= \sup_{ \vep,\gamma_2 \in (0,\beta_0]} Q(\uxgbe,\oxgbe) <\infty.
\ee

\noindent
 {\emph{Step 4.a}} ($\vep \downarrow 0$). By their definitions,
 Cauch-Schwarz inequality, and \eqref{eq:summable},
\begin{align*}
|\underline{\theta}^\gamma_\vep-\overline{\theta}^\gamma_\vep|
&=|\Theta^\gamma_\vep(\uxgbe)-\Theta^\gamma_\vep(\oxgbe)|
 \le \frac{1}{\vep}\, \sum_k\, |(\umogbe-\omogbe)(f_k)|\, |f_k(\uxgbe)-f_k(\oxgbe)|\\
& \le \frac{1}{\vep}\, \sum_k\, |(\umogbe-\omogbe)(f_k)|\, \|f_k\|_{\sC^1} |\uxgbe - \oxgbe|\\
&\le  \frac{1}{2\vep}\, \sum_k\,  |(\umogbe-\omogbe)(f_k)|^2\, +\,
\frac{1}{2\vep}\,  |\uxgbe - \oxgbe|^2 \, \sum_k\,   \|f_k\|_{\sC^1}^2\\
& \le \frac{1}{2\vep}\, (\varrho^2(\umogbe -\omogbe) + |\uxgbe-\oxgbe|^2) 
= \frac{1}{2\vep}\, \rho^2(\umogbe -\omogbe,\uxgbe-\oxgbe).
\end{align*}
Additionally, writing $\rho_{\vep,\gamma} = \rho(\umogbe -\omogbe,\uxgbe-\oxgbe)$, we have
\begin{align*}
\rho_{\vep,\gamma} |\underline{\theta}^\gamma_\vep|
&= \frac{1}{2\vep} \rho_{\vep,\gamma}^2
+ \frac{\vep}{2} |\underline{\theta}^\gamma_\vep|^2
= \frac{1}{2\vep}  \rho_{\vep,\gamma}^2
+ \frac{1}{2\vep} 
\Big(\sum_k (\umogbe -\omogbe)(f_k) \, f_k(\uxgbe)\, \Big )^2\\[0.5em]
& \le \frac{1}{2\vep} \rho_{\vep,\gamma}^2
+ \frac{1}{2\vep} \varrho^2(\umogbe-\umogbe).
\end{align*}
Same estimate also holds for $\rho_{\vep,\gamma} |\overline{\theta}^\gamma_\vep|$.
Proceeding similarly, we obtain
$$
\rho_{\vep,\gamma} |\Delta_\vep| \le  
\frac{1}{2\vep}\rho_{\vep,\gamma}^2 
+\frac{1}{2\vep}\  |\uxgbe-\oxgbe|^2.
$$

In view of \eqref{eq:limzero}, \eqref{eq:Qge},
and the above inequalities,
$\lim_{\vep \downarrow 0}\, \zeta^\gamma_\vep =0$,
for every $\gamma_1,\gamma_2 >0$.
Then, we take the limit as $\vep \downarrow  0$ in the inequality \eqref{eq:H}.
Since $m$ in \eqref{eq:H} is continuous with $m(0)=0$, 
the limit of \eqref{eq:H} is the following, 
$$
0 \le - c_0 \beta_0 + C_* \, \lim_{\vep \downarrow 0}\, [\gamma_1 Q^2(\uxgbe,\oxgbe)
+\gamma_2 Q(\uxgbe,\oxgbe)].
$$

\noindent
{\emph{Step 4.b}} ($\gamma_1 \downarrow 0$, then $\gamma_2 \downarrow 0$).
In view of \eqref{eq:Qge},
$\lim_{\gamma_1 \downarrow 0}( \gamma_1 Q^2(\uxgbe,\oxgbe))=0$. Hence,
$$
0 \le - c_* \beta_0 + C_* \lim_{\gamma_1 \downarrow 0} \lim_{\vep \downarrow 0}  
 [\gamma_1 Q^2(\uxgbe,\oxgbe)
+\gamma_2 Q(\uxgbe,\oxgbe)]
= - c_* \beta_0 + C_* \lim_{\gamma_1 \downarrow 0} \lim_{\vep \downarrow 0}  
\gamma_2 \, Q(\uxgbe,\oxgbe).
$$
We now let $\gamma_2$ tend to zero and use \eqref{eq:limzero} to 
conclude that
$$
0 \le - c_* \beta_0 + C_* \lim_{\gamma_2 \downarrow 0}\, \lim_{\gamma_1 \downarrow 0} \, \lim_{\vep \downarrow 0}  
\, \gamma_2 \, Q(\uxgbe,\oxgbe) =-c_* \beta_0.
$$
As $\beta_0, c_*>0$ this clear contradiction implies
that \eqref{eq:counter} cannot hold completing the proof.
\qed

\section{Examples}
\label{sec:examples}
\subsection{Heat Equation}\label{sec:heat}
Let $X$ be a $d-$dimensional, uncontrolled  Brownian motion, that is $b\equiv 0$ and $\sigma = I\in \SSS^d$ in  \eqref{eq:OSDE}. 
We consider the standard time $\Lambda_t = t$  in \eqref{eq:occFlow} (or equivalently   $\Lambda_t = \frac{1}{d}\text{tr}(\langle X \rangle_t)$ since $X$ is Brownian motion) and set  $\ell \equiv 0$. Given  $g \in \sC_b(\partial \sD_T)$,  
 the value function \eqref{eq:valueFct}  is given by 
\begin{equation}
    v(\mo,x) = \E^{\Q}_{\mo,x}[g(\calO_{T-|\mo|},X_{T-|\mo|})] = \E^{\Q}[g(\mo + \calO_{T-|\mo|}^x,x  +X_{T-|\mo|})], 
\end{equation}
with $(\mo,x)\in \sD_T$ and the shifted measure $\calO_{s}^x(B) = \calO_{s}(B-x)$, $B\in \sB(\R^d)$.  
The associated Hamiltonian is $\sH(\mo,x,\theta,\Delta,\Gamma) = -\theta  - \frac{1}{2}\text{tr}(\Gamma)$, and  the dynamic programming equation coincides with 
the (backward) occupied heat equation
\begin{subnumcases}{}
-\partial_\mo u - \frac{1}{2}\triangle u = 0 & \text{ on }   
$\mathring{\sD_T}$, \label{eq:Heat1}\\
      u = g, & \text{ on }  $\partial  \sD_{T} $, 
      \label{eq:Heat2} 
  \end{subnumcases}
and  the Laplacian $\triangle = \sum_{i=1}^d \partial_{x_ix_i}$. 

  \begin{remark}\label{rem:PPDEHeat} Recall the space of continuous paths $\Omega = \calC([0,T];\R^d)$  seen in the Introduction. 
    It is well-known that the path-dependent heat equation $$-\partial_tu(t,\omega) - \frac{1}{2}\textnormal{tr}(\partial_{\omega}^2 u (t,\omega)) = 0, \quad u(T,\cdot) = G \in \calC(\Omega),$$ 
    does not always admit a classical solution.  One example is $G(\omega)= \omega_{t_0 }$ for some fixed time $t_0 <T$; see \cite[Chapter 11]{ZhangBook}. 
    We note that functionals depending on specific values of the path before the final time $T$ cannot be expressed in terms of the occupation measure $\cO_T(\omega)$ since the latter  erases the chronology of $\omega$. 
    Hence, it is not possible to translate the above  counterexample to occupied PDEs. 
    In fact, 
    we
    conjecture that  the occupied heat equation \eqref{eq:Heat1}-\eqref{eq:Heat2} always admits a classical solution whenever the terminal functional $g:\partial\sD_T\to \R$ is continuous with respect to the weak $\times$ Euclidean topology. 
\end{remark} 

  We now discuss some examples. 

     \begin{example} 
Let $g(\mo,x)  = \mo(B)$ with $B = \{|x| \le 1\}$.  Then $g(\calO_t,X_t) = \int_0^t\mathds{1}_B(X_s)ds$ is the occupation time of Brownian motion in the unit ball. While $g$ is
continuous in $\mo$ with respect to the stronger total variation distance,
 it is not continuous with respect to the weak topology and the corresponding path functional is not continuous either. 
 The value function is given by 
$$v(\mo,x) = \mo(B) + \E^{\Q}[\calO_{T-|\mo|}^x(B)] =  \mo(B) + \int_0^{T-|\mo|}\Q(|x+X_{s}|\le 1)ds. $$  
    Observe that the occupation derivative  
    $\partial_{\mo}v(\mo,x) = \mathds{1}_B(x) \ - \ \Q(|x+X_{T-|\mo|}|\le 1)$ is discontinuous in $x$.  
      Hence   the occupied heat equation 
  \eqref{eq:Heat1}$-$\eqref{eq:Heat2} does \textit{not} admit a classical  solution, motivating the viscosity theory developed throughout. 
 \end{example}



\begin{example}\label{ex:HeatLinear2}
{Let $g(\mo,x)  = \psi(\mo(\phi))$ for some 
 functions $\phi\in \sC(\R^d)$, $\psi\in \sC(\R)$.
  $X$ is again  a Brownian motion.}
Then, $v(\mo, x) = u(|\mo|, x, \mo(\phi))$, where
  $$
  u(t,x,y) := \E^{\Q}\Big[\psi\Big( y + \int_0^{T-t}\phi(x+X_s)ds\Big)\Big].
  $$
  It is clear that $u$ is the unique viscosity solution of the following degenerate 
 parabolic linear PDE
$$
-\partial_t u - \frac12\triangle_{x} u - \phi(x) \partial_y u=0,\quad u(T,x,y) = \psi(y).
 $$
 When $\psi$ is smooth,  $u$ is also smooth and
  $$
 \partial_{\mo}v(\mo,x) = \E^{\Q}\Big[\psi'\Big( o(\phi) 
 + \int_0^{T-|o|}\phi(x+X_s)ds\Big)\Big(\phi(x) -\phi(x+X_{T-|\mo|})\Big)\Big].
 $$
We additionally have that 
$$
\partial_{\mo}v(\mo,x) = \partial_t u(|\mo|, x, \mo(\phi)) 
+ \partial_y u(|\mo|, x, \mo(\phi)) \phi(x) = -\frac12 \triangle_{x} u(|\mo|, x, \mo(\phi)).
$$
 
{However, when $\psi$ is not smooth, in general $u$ is not smooth either. 
This on the other hand does not immediately imply that $v$ is also not smooth.
To illustrate this point,
suppose that $\phi \equiv 1$.  Then,  $u(t,x,y) = \psi(y + T-t)$,
is not smooth.  But as $\mo(\phi)=|\mo|$, 
$
v(\mo,x) =u(|\mo|, x, |\mo|)= \psi(T),
$
and hence, $v$ is obviously smooth. 
This is due to the fact that the existence of 
$\partial_{\mo} v$ depends on the smoothness
of  the combined operator $\partial_t u+ \phi(x)\partial_y u$,
and not each one separately. A related discussion can be found in  \cite{BouchardTan} in the case $\phi(x) = x$. }
\end{example}
}

\subsection{Timer options}

Timer options  \cite{carole,GuyonHL},   are contingent claims where the maturity is floating and depends on the realized variance of the underlying asset. More precisely, the buyer of the option receives the payoff as soon as the realized variance exceeds a given threshold which we denote by $T>0$. 
We now derive the pricing PDE of timer options written on a single asset, i.e. $d=1$.  See also Section 5.2 in \cite{Tissot}.  Suppose that the log  price of the asset and its occupation flow evolve according to the OSDE (assuming zero interest rate) 
    $$d\calO_t =  \sigma^2_t 
  \ \delta_{X_t}\ d t , \quad  dX_t = -\sigma_t^2/2dt + \sigma_tdW_t, \quad \sigma_t = \sigma(\calO_t,X_t).$$
  We therefore choose $\lambda(\mo,x,a) = \sigma(\mo,x)^2$ in \eqref{eq:OSDEOcc}. 
Given $g:\partial\sD_T\to\R$, consider  a 
 timer option that pays $g(\cO_{\tau},X_{\tau})$ at  $\tau = \inf \{t\ge 0 : |\calO_t| \ge T \}$.  Since $ |\calO_t| = \langle X \rangle_t = \int_0^t \sigma_s^2 ds$ is the realized variance of the asset up to time $t$, the maturity $\tau$ indeed corresponds to the first time the  realized variance exceeds the  budget $T$. Note also that the option is \textit{exotic} 
 in the sense that the payoff $g(\cO_{\tau},X_{\tau})$ depends on the path $(X_t)_{t\le \tau}$ (through $\cO_{\tau}$). 
 For instance,  a fixed strike  Asian put option is obtained by setting 
$$ g(\mo,x) = \left(K - \frac{1}{|\mo|}\int_{\R} y \mo(dy) \right)^+\; \Longrightarrow \; g(\cO_{\tau},X_{\tau}) = \left(K - \frac{1}{\tau}\int_0^{\tau} X_tdt \right)^+. $$
See \cite[Section 5]{Tissot} for a  list of exotic payoffs expressible in terms of the occupied process. In light of the control problem \eqref{eq:controProblem}, we set $\ell \equiv 0$,   so the value function corresponds to the price of the option, namely
\begin{equation*}
    v(\mo,x) =  
    \E^{\Q}_{\mo,x}
[g(\cO_{\tau},X_\tau)].
\end{equation*} 
The associated  Hamiltonian reads 
\begin{equation}\label{eq:HamiltonianTimer}
    \sH(\mo,x,\theta,\Delta,\Gamma) =  -\Big(\sigma^2(\mo,x)  \theta 
     - \frac{1}{2}\sigma^2(\mo,x) \Delta + \frac{1}{2}\sigma^2(\mo,x) \Gamma\Big), 
\end{equation} 
     with the \textit{Greeks} $\theta = \partial_{\mo}v(\mo,x)$, $\Delta = \partial_{x}v(\mo,x)$, and $\Gamma = \partial_{xx}v(\mo,x)$. 
     We therefore conclude from \cref{thm:main} that the price function $v$ is the unique viscosity solution of \eqref{eq:HJB12UVM}--\eqref{eq:HJB22UVM}. 
       Moreover, observe that the square volatility in the Hamiltonian \eqref{eq:HamiltonianTimer} can be factored out and  due to the  ellipticity condition \eqref{eq:WeakEllipticity}, the pricing PDE can be simplified to 
 \begin{subnumcases}{}
 -\partial_{\mo}v +  \frac{1}{2} \partial_{x}v - 
  \frac{1}{2}\partial_{xx}v
 = 0 & \text{ on }   
$\mathring{\sD_T}$, \label{eq:HJB12UVM}\\
      v = g, & \text{ on }  $\partial  \sD_{T} $.
      \label{eq:HJB22UVM} 
  \end{subnumcases}
In other words, the volatility used to price the option  is here irrelevant. 
This exhibits an important virtue of timer options, namely is that they do not entail any model risk with respect to volatility.  We note that this is no longer true when interest rates are nonzero; see \cite{carole}. 
  
\subsection{Exotic Options in the Uncertain Volatility Model  }
Following the uncertain volatility model introduced by \citet{Avellaneda}, we consider an  asset price $X^{\alpha}$ with dynamics $d X_t^{\alpha} = 
  \alpha_t X_t^{\alpha} \d W_t.$  
The control process $\alpha$
 represents the  volatility of the asset. The control set is $A = [\underline{a},\overline{a}] \subset  (0,\infty)$ meaning that volatility, albeit uncertain, is kept in a compact interval. In the occupied SDE \eqref{eq:OSDE}, we thus choose   $b\equiv 0$ and 
$\sigma(\mo,x,a) = ax$, so that  $\sigma$ satisfies the 
growth and Lipschitz conditions with constant $c_* = \overline{a}$.  

Consider an exotic option with payoff  $\varphi = \varphi(\cO_T^{\alpha},X_T^{\alpha})$ and assume calendar time for $\calO^{\alpha}$, i.e. $\lambda \equiv 1$. Observe that $J(\mo,x,\alpha) = \E^{\Q}_{\mo,x}
[\varphi(\cO_{T - |\mo|}^{\alpha},X_{T - |\mo|}^{\alpha})]$ is the price of the option under the volatility process $\alpha$. 
Suppose we are interested in the seller's price  $p:\sD_T\to\R$ of the option, which  is obtained  from the volatility $\alpha$ that \textit{maximizes} $J(\mo,x,\alpha)$. 
We therefore set $\ell \equiv 0$, $g = -\varphi$ in \eqref{eq:controProblem}, 
so  the price of the option  coincides with  the negative of the value function.  That is,  
\begin{equation*}
    p(\mo,x) = -v(\mo,x) =  
    \sup_{\alpha \in \sA}
    \E^{\Q}_{\mo,x}
[\varphi(\cO_{T - |\mo|}^{\alpha},X_{T - |\mo|}^{\alpha})].
\end{equation*}
The associated dynamic programming equation is 
\begin{subnumcases}{}
 \partial_{\mo}p + \sup_{a \in A}  
  \frac{1}{2}a^2 x^2\partial_{xx}p 
 = 0 & \text{ on }   
$\mathring{\sD_T}$, \label{eq:HJB12UVMa}\\
      p = \varphi, & \text{ on }  $\partial  \sD_{T} $,
      \label{eq:HJB22UVMb} 
  \end{subnumcases}
where \eqref{eq:HJB12UVMa} can be rewritten  as the (occupied) Black-Scholes-Barenblatt equation  \cite{Avellaneda,GuyonHL},
\begin{equation}
    \partial_{\mo}p +  \frac{1}{2}x^2V( 
     \partial_{xx}p)
 = 0,
\end{equation}
with $V(\Gamma) = (\underline{a}^2\mathds{1}_{\{\Gamma <0\}} + \overline{a}^2\mathds{1}_{\{\Gamma\ge 0\}} )\Gamma$. 
Applying \cref{thm:main} to $v = -p$, it is then immediate  to see that  $p$ is the unique viscosity solution of \eqref{eq:HJB12UVMa}--\eqref{eq:HJB22UVMb}. 

\appendix

\section{Proof of \cref{prop:Ito}}
\label{sec:Itoproof}
\begin{proof} 
  Given $t\ge 0$, let   $0=t_0 < \cdots < t_N = t$ be a sequence of partitions of $[0,t]$ that satisfies    $\max_{n\le N}|t_{n} - t_{n-1}|\to 0$
  as $N\to \infty$. Consider  the piecewise constant process 
  \begin{equation}\label{eq:piecewiseconstantX}
  X^{N}_s = X_0\mathds{1}_{\{s = 0\}}+\sum_{n} X_{t_{n}}\mathds{1}_{(t_{n-1},t_{n}]}(s), \quad  s \le t.
  \end{equation}
Write also $\Lambda_t = \int_0^t\lambda_sds$ and introduce the discretized occupation flow
        $\cO^{N}_{s}  = \int_0^s \delta_{X_u^N} d\Lambda_u$. 
 We note that here $X^N$ and $\cO^N$ are not $\F$-adapted, however, $X^N_{t_n}$, $\cO^N_{t_n}$ are $\cF_{t_n}$-measurable, which we will actually use. We first establish that 
 \begin{align}
 \label{ONconv}
     \lim_{N\to \infty}\sup_{0\le s\le t}\varrho( \cO_{s} - \cO^{N}_{s}) = 0,\quad \Q\mbox{-almost surely}.
 \end{align} 
  For each $k\in \N$ and $f_k$ introduced in \cref{ssec:cylindrical}, we compute 
  \begin{align*}
     |(\cO_{s} - \cO^N_{s})(f_k)|^2 &= \Big|\int_0^s (f_k(X_u) -  f_k(X_u^N))d\Lambda_u\Big|^2  \le \Lambda_s\int_0^s | f_k(X_u) -  f_k(X_u^N)|^2d\Lambda_u \\[0.5em]
     &\le  \Lambda_s  \int_0^s \lVert f_k\rVert_{\sC^1}^2 |X_u -  X_u^N|^2d\Lambda_u.
  \end{align*}
  As the last term  is nondecreasing in $s$ for all $k$, 
  we can  set $s=t$ to obtain 
    \begin{align*}
     \sup_{s\le t} \varrho( \cO_{s} - \cO^{N}_{s})^2 
     \le  \sum_{k} \Lambda_t\int_0^t \lVert f_k\rVert_{\sC^1}^2|X_u -  X_u^N|^2d\Lambda_u \le \Lambda_t\int_0^t |X_u -  X_u^N|^2d\Lambda_u.  
  \end{align*}
For $\Q-$almost all $\omega \in \Omega$, $s \mapsto X_s(\omega)$ is uniformly continuous on the compact $[0,t]$, hence it admits a  modulus of continuity $\frakm:\R_+\to\R_+$ (which depends on $\omega$). 
Consequently,
    \begin{align*}
     \int_0^t |X_u -  X_u^N|^2d\Lambda_u = \sum_{n\le N} \int_{t_{n-1}}^{t_n} |X_u -  X_{t_{n-1}}|^2d\Lambda_u  
     \le  \frakm\! \big(\max_{n\le N}|t_{n} - t_{n-1}|\big)^2 \Lambda_t \to 0, \quad N\to \infty,  
  \end{align*}
  which proves \eqref{ONconv}.
 Together with the continuity of $v$, we also have   that 
 \begin{align}
 \label{vONconv}
   \lim_{N\to \infty}\sup_{0\le s\le t}\big| v(\calO_{s}^N,X_{s}) - v(\calO_{s},X_{s}) \big| =0.
 \end{align}
 Next, we write 
\begin{equation}\label{eq:telescopic}
      v(\calO_{t}^N,X_{t}) -  v(\calO_{0}^N,X_{0}) 
           = \sum_{n\le N}  ( v(\calO_{t_{n}}^{N},X_{t_{n}}) -  v(\calO_{t_{n-1}}^N,X_{t_{n-1}}) ).
 \end{equation}
 Denote $\Delta^{\! \Lambda}_n := \Lambda_{t_n} - \Lambda_{t_{n-1}}$. Then
\begin{align*}
  &v(\calO_{t_{n}}^{N},X_{t_{n}}) -  v(\calO_{t_{n-1}}^N,X_{t_{n-1}})\\[0.25em]
  &= \big[v(\calO_{t_{n-1}}^{N}+ \Delta^{\! \Lambda}_n \delta_{X_{t_n}},X_{t_{n}}) -  v(\calO_{t_{n-1}}^{N},X_{t_{n}})\big] + \big[v(\calO_{t_{n-1}}^{N},X_{t_{n}}) - v(\calO_{t_{n-1}}^N,X_{t_{n-1}}) \big]\\
  &= \int_0^1 \partial_\mo v(\calO_{t_{n-1}}^{N}+ h\Delta^{\! \Lambda}_n\delta_{X_{t_n}},X_{t_{n}}) dh ~\Delta^{\! \Lambda}_n\\
  & + \int_{t_{n-1}}^{t_n} \Big[\nabla v(\calO_{t_{n-1}}^{N},X_s) \cdot dX_s + \frac{1}{2}\text{tr}(\nabla^2 v(\calO_{t_{n-1}}^N,X_s)\sigma_s \sigma_s^{\!\top})ds\Big].
\end{align*}
Here in the last equality, the first term is due to the definition and  continuity of $\partial_\mo v$, and the second term is due to the classical It\^{o} formula.  
Then we can rewrite  \eqref{eq:telescopic} as 
  \begin{align}
       & v(\calO_{t}^N,X_{t}) -  v(\calO_{0}^N,X_{0}) 
           = \sum_{n\le N} \int_0^1 \partial_\mo v(\calO_{t_{n-1}}^{N}+ h\Delta^{\! \Lambda}_n\delta X_{t_n},X_{t_{n}}) dh ~\Delta^{\! \Lambda}_n
           \nonumber\\[0.25em] 
        &\quad+ \sum_{n\le N} \int_{t_{n-1}}^{t_n} \Big[\nabla v(\calO_{t_{n-1}}^{N},X_s) \cdot dX_s + 
        \frac{1}{2}\text{tr}(\nabla^2 v(\calO_{t_{n-1}}^N,X_s)\sigma_s \sigma_s^{\!\top})ds\Big]. 
        \label{eq:dIto} 
   \end{align}
   Similarly to \eqref{ONconv}, one can easily show that 
   \begin{align*}
      \lim_{N\to\infty} \max_n \sup_{h\in [0,1]} \sup_{t_{n-1}\le s\le t_n} \rho\Big(\calO_{t_{n-1}}^{N}+ h\Delta^{\! \Lambda}_n \delta_{X_{t_n}} - \calO_s, X_{t_n}-X_s\Big)=0.
   \end{align*}
   Then, by \eqref{vONconv}, \eqref{eq:dIto}, and the regularity of $v$, we obtain \eqref{eq:ito} immediately. 
\end{proof}

\section{Proof of \cref{prop:Lipschitz}}
\label{sec:Lipschitz}

      We  show that $J(\mo,x,\alpha)$ is locally $1/2-$Hölder continuous   with constant $\hat{c}$ (to be determined) independent of $\alpha \in \sA$. It is then immediate to see that  $v(\mo,x)$ is  locally $1/2-$Hölder continuous  with the same constant.  
    Fix $(\mo,x),(\mo',x') \in \sD_T$, $\alpha \in \sA$, and assume without loss of generality that $\rho(\mo-\mo',x-x')\le 1$, i.e.,  $\delta = 1$ in the statement. Write $(\cO,X)$, $(\cO',X')$ for  the solution of \eqref{eq:OSDEOcc}$-$\eqref{eq:OSDE} controlled by $\alpha$ with  initial value $(\mo,x),(\mo',x')$, respectively. The exit times of $(\cO,X)$, $(\cO',X')$ from $\mathring{\sD_T}$ are  denoted by $\tau,\tau'$.  We also set  $\varphi_t = \varphi(\cO_t,X_t,\alpha_t)$, $\varphi_t' = \varphi(\cO_t',X_t',\alpha_t)$, $\varphi \in \{\lambda,b,\sigma,\ell\}$ and $\Lambda = \int_0^{\cdot} \lambda_sds$, $\Lambda' = \int_0^{\cdot} \lambda_s'ds$.  Due to  
    the triangle inequality   
\begin{align*}
   |J(\mo,x,\alpha)  - J(\mo,x',\alpha)| 
   \le \E^{\Q}[|\int_0^{\tau}\ell_sds - \int_0^{\tau'}\ell_s'ds|] + \E^{\Q}[|g(\cO_{\tau},X_{\tau}) - g(\cO_{\tau'}',X_{\tau'}')|],
\end{align*}
 we can estimate the terminal and running cost  separately. 

\subsection{Terminal Cost}

 Using the Lipschitz condition satisfied by $g$, then 
\begin{align*}
 \E^{\Q}[|g(\cO_{\tau},X_{\tau}) - g(\cO_{\tau'}',X_{\tau'}')|]
   \le c_* \E^{\Q}[\rho(\cO_{\tau}-\cO_{\tau'}',X_{\tau}-X_{\tau'}')],
\end{align*}
with the parabolic norm $\rho(\mo,x) = \sqrt{\varrho(\mo)^2 + |x|^2}$.  Next, consider  the decomposition, 
\begin{equation}\label{eq:appBDecomp}
    \rho(\cO_{\tau}-\cO_{\tau'}',X_{\tau}-X_{\tau'}') \le \rho(\cO_{\tau}-\cO_{\tau'},X_{\tau}-X_{\tau'}) + \rho(\cO_{\tau'}-\cO_{\tau'}',X_{\tau'}-X_{\tau'}').
\end{equation}
First, we estimate the rightmost term in  \eqref{eq:appBDecomp}. 

\begin{lemma} \label{lem:Gronwall} 
There exists a positive constant $C_1$ such that 
   \begin{equation}
        \E^{\Q}[\rho(\cO_{\tau'}-\cO_{\tau'}',X_{\tau'}-X_{\tau'}')^2] \le  C_1\rho(\mo-\mo',x-x')^2.
   \end{equation}
\end{lemma}
\begin{proof}
Recalling that  condition \eqref{eq:WeakEllipticity}  implies  $\tau' \le c_*T =: T_*$, then
\begin{equation}\label{eq:appBBound}
    \E^{\Q}[\rho(\cO_{\tau'}-\cO_{\tau'}',X_{\tau'}-X_{\tau'}')^2] \le R(T_*), \quad R(t) := \E^{\Q}[\sup_{s\le t\wedge \tau'}\rho(\cO_{s}-\cO_s',X_s-X_s')^2]. 
\end{equation}
Fix $t\le T_*$. From 
Cauchy-Schwarz and Doob inequalities, it is classical that 
\begin{align*}
    \frac{1}{3}\E^{\Q} [ \ \sup_{s\le t\wedge \tau' }|X_s-X_s'|^2 ] &\le  |x-x'|^2 + t\E^{\Q}[\int_0^{t\wedge \tau'} |b_s - b_s'|^2ds] + 4 \E^{\Q}[\int_0^{t\wedge \tau'}|\sigma_s - \sigma_s'|^2ds] \\
    &\le |x-x'|^2 + c_*^2(t+4)\int_0^t\E^{\Q}[\sup_{u\le s\wedge \tau'} \rho(\cO_{u}-\cO_u',X_u-X_u')^2]ds\\
    &\le |x-x'|^2 + c_*^2(T_*+4)\int_0^tR(s)ds.
\end{align*}
Moreover, for each $k\in \N$, the triangle inequality and occupation time formula  yields 
\begin{align}\label{eq:appBvarrho}
    \sup_{s\le t\wedge \tau'}|(\cO_{s}-\cO_s')(f_k)| &\le |(\mo-\mo')(f_k)| + \int_0^{t\wedge \tau'} | f_k(X_s)\lambda_s -  f_k(X_s')\lambda_s'|ds.
\end{align}
Next,  use $|f\lambda - f'\lambda'| \le |f||\lambda-\lambda'| + |f - f'| |\lambda'|$   to obtain 
\begin{align}
  | f_k(X_s)\lambda_s -  f_k(X_s')\lambda_s'| &\le |f_k(X_s)||\lambda_s-\lambda_s'|  + |f_k(X_s) - f_k(X_s')| \lambda_s'  \nonumber\\[0.5em]
    &\le \lVert f_k\rVert_{\infty}|\lambda_s-\lambda_s'| + \lVert \nabla f_k\rVert_{\infty}|X_s - X_s'| \lambda_s'   \nonumber\\[0.5em]
    &\le c_*\lVert f_k\rVert_{\sC^1} (1+ \lambda_s')\rho(\cO_{s}-\cO_s',X_{s}-X_s'),\label{eq:appBintegrand}
\end{align}
using the Lipschitz property of $\lambda$ in the last inequality. 
Using $\int_0^{t\wedge \tau'}\lambda_s' ds \le  \Lambda_{\tau'}' = |\cO_{\tau'}'| = T$,   $\tau'\le T_*$, and  \eqref{eq:appBvarrho}, we obtain by integrating \eqref{eq:appBintegrand} that 
$$ \frac{1}{2}\sup_{s\le t\wedge \tau'}|(\cO_{s}-\cO_s')(f_k)|^2 \le |(\mo-\mo')(f_k)|^2 + c_*^2\lVert f_k\rVert_{\sC^1}^2 (T_*  + T)^2\int_0^t\sup_{u\le s\wedge \tau'} \rho(\cO_{u}-\cO_u',X_u-X_u')^2ds.$$
summing over $k$ and using  $\sum\nolimits_k\lVert f_k\rVert_{\sC^1}^2\le 1$ leads to 
\begin{align*}
  \frac{1}{2}\E^{\Q}[\sup_{s\le t\wedge \tau'}\varrho(\cO_{s}-\cO_s')^2]  
  \le \varrho(\mo - \mo')^2 +   c_*^2 (T  + T_*)^2  \int_0^t R(s) ds.
\end{align*}
Altogether, we have established  that 
\begin{align*}
    R(t)&\le \E^{\Q}[\sup_{s\le t\wedge \tau'}\varrho(\cO_{s}-\cO_s')^2] + \E^{\Q}[\sup_{s\le t\wedge \tau'}|X_{s}-X_s'|^2] 
    \le 3\rho(\mo-\mo',x-x')^2 + C\int_0^{t} R(s) ds
\end{align*}
where $C = c_*^2[3(T_*+4) + 2(T+T_*)^2]$. We therefore conclude from Grönwall's Lemma that 
$$ \E^{\Q}[\rho(\cO_{\tau'}-\cO_{\tau'}',X_{\tau'}-X_{\tau'}')^2] \le R(T_*) \le C_1\rho(\mo-\mo',x-x')^2, \qquad C_1 = 3e^{C  T_*}.$$ 
    
\end{proof}

We now  treat the remaining term in \eqref{eq:appBDecomp}, namely $\rho(\cO_{\tau}-\cO_{\tau'},X_{\tau}-X_{\tau'})$. First observe that 
\begin{equation}\label{eq:appB2ndTerm}
    \E^{\Q}[\rho(\cO_{\tau}-\cO_{\tau'},X_{\tau}-X_{\tau'})] \le  \E^{\Q}[\varrho(\cO_{\tau}-\cO_{\tau'})]+ \E^{\Q}[|X_{\tau}-X_{\tau'}|].
\end{equation}
 Let $\underline{\tau} := \tau \wedge \tau'$ and $\overline{\tau} := \tau \vee \tau'$. Using the linear growth of $\lambda$,  then for all $k\ge 1$, 
 \begin{align*}
    \frac{1}{2}[|(\cO_{\tau}-\cO_{\tau'})(f_k)|^2 - |(\mo-\mo')(f_k)|^2] \le    \big(\int_{\underline{\tau}}^{\overline{\tau}} | f_k(X_s)|\lambda_sds \big)^2
    \le \lVert f_k\rVert_{\sC^1}^2 Z_*^2     |\tau-\tau'|^2 ,
\end{align*}
where $Z_* = c_*\sup_{t\le T_*}(1+\rho(\cO_t,X_t)) $. 
Summing over $k$ and rearranging yields
\begin{equation}\label{eq:boundVARRHO}
   \E^{\Q}[\varrho(\cO_{\tau}-\cO_{\tau'})^2] \le 2\varrho(\mo-\mo')^2 + 2C^2 \rVert\tau-\tau'\rVert_{L^2(\Q)}^2,
\end{equation}
with the finite constant $C = \lVert Z_* \rVert_{L^2({\Q})} $. For the rightmost term in \eqref{eq:appB2ndTerm}, we first note that 
\begin{align*}
    \E^{\Q}[|X_{\tau}-X_{\tau'}|] &\le \E^{\Q}\big[|\int_{\underline{\tau}}^{\overline{\tau}} b_sds|\big] + \E^{\Q}\big[|\int_{\underline{\tau}}^{\overline{\tau}}\sigma_sdW_s|\big]. 
\end{align*}
Using  the linear growth of $b,\sigma$ and Burkholder-Davis-Gundy (BDG) inequality \cite[Chapter IV]{RevuzYor}, it follows that 
\begin{align}
\E^{\Q}\big[|\int_{\underline{\tau}}^{\overline{\tau}} b_sds|\big] &\le \E^{\Q}\big[\int_{\underline{\tau}}^{\overline{\tau}} |b_s|ds\big] \le  \E^{\Q}[Z_*|\tau-\tau'|]\le C \rVert\tau-\tau'\rVert_{L^2(\Q)}, \label{eq:bIneq}\\[0.5em]
\E^{\Q}\big[|\int_{\underline{\tau}}^{\overline{\tau}}\sigma_sdW_s|\big] 
&\le C_1^{\textsc{\tiny BDG}}\E^{\Q}\big[ \lVert \sigma \rVert_{L^2([\underline{\tau},\overline{\tau}])}\big] \le C_1^{\textsc{\tiny BDG}}\E^{\Q}[Z_*|\tau-\tau'|^{1/2} ]\le \tilde{C} \ \rVert\tau-\tau'\rVert_{L^1(\Q)}^{1/2}, \label{eq:sigIneq} 
\end{align}
with $\tilde{C} = C C_1^{\textsc{\tiny BDG}}$. 
We now prove the following lemma.
\begin{lemma}\label{lem:tau}
There exists  $C_2>0$ such that $
    \rVert\tau-\tau'\rVert_{L^2(\Q)} \le C_2\rho(\mo-\mo',x-x')$. 
\end{lemma}
\begin{proof} 
Let $\Delta_*^{\! \Lambda} :=\sup_{t\le T_*} |\Lambda_{t} - \Lambda_{t}'|$, where we recall that $\Lambda$ (respectively $\Lambda'$) coincides with the total mass process $|\calO|$ (resp. $|\calO'|$). Since $\Lambda_{\tau'}' = T$, then $\Lambda_{\tau'} = \Lambda_{\tau'}' + (\Lambda_{\tau'} - \Lambda_{\tau'}')\ge T - \Delta_*^{\! \Lambda}$. 
Moreover, the nondegeneracy condition \eqref{eq:WeakEllipticity} implies that   
$$ \Lambda_{\tau' + s} = \Lambda_{\tau'} + \int_{\tau'}^{\tau'+s}\lambda_{u}du \ge T-\Delta_*^{\! \Lambda} + s/c_*, \quad \forall s\ge 0.$$ 
Choosing $s = c_*\Delta_*^{\! \Lambda}$ therefore gives $\tau \le \tau'+c_*\Delta_*^{\! \Lambda}$. Similarly,  $\tau' \le \tau+c_*\Delta_*^{\! \Lambda}$, which shows that 
    \begin{equation} \label{eq:tauIntermediate}
    |\tau-\tau'| \le c_*\sup_{t\le T^*}  |\Lambda_{t} - \Lambda_{t}'|,  \qquad \Q-a.s.
\end{equation}
Next, noting that $\Lambda_0 = |\mo|$, $\Lambda_0' = |\mo'|$, then for all $t\le T_*$, $   |\Lambda_t - \Lambda_t'| \le  ||\mo| -|\mo'|| + \int_0^t|\lambda_s-\lambda_s'|ds. $ 
Recalling that the first element of the separating family $(f_k)_{k\in \N}$ is constant, say $f_0 \equiv c_0$ with $c_0\in (0,1)$, we have  
$$||\mo| - |\mo'|| =  |\int_{\R^d}(\mo -\mo')(dx)|  = c_0^{-1}|(\mo -\mo')(f_0)|\le c_0^{-1}\varrho(\mo -\mo').$$
Moreover, we have from the Lipschitz continuity of $\lambda$ that 
\begin{align*}
    \int_0^t|\lambda_s-\lambda_s'|ds \le c_*\int_0^t\rho(\cO_s - \cO_s',X_s - X_s')ds \le c_*T_* \sup_{s\le T_*}\rho(\cO_s - \cO_s',X_s - X_s').
\end{align*}
Using similar arguments as in  the proof of  \cref{lem:Gronwall}, we obtain  
\begin{equation*}
    \frac{1}{2}\E^{\Q}[\sup_{t\le T_*}|\Lambda_t - \Lambda_t'|^2] \le  c_0^{-2}\varrho(\mo -\mo')^2 + c_*^2T_*^2C_1\rho(\mo-\mo',x-x')^2.
\end{equation*}
Together with \eqref{eq:tauIntermediate}, this yields 
$$\rVert\tau-\tau'\rVert_{L^2(\Q)}\le c_*\rVert\sup_{t\le T^*}  |\Lambda_t - \Lambda_t'|\rVert_{L^2(\Q)}\le C_2\rho(\mo-\mo',x-x'), \quad C_2 =c_*\sqrt{2(c_0^{-2} + c_*^2T_*^2C_1)}.$$ 
\end{proof}
Combining the above Lemma with equations \eqref{eq:appB2ndTerm}, \eqref{eq:boundVARRHO} yields
\begin{equation*}
    \E^{\Q}[\varrho(\cO_{\tau}-\cO_{\tau'})] \le
     (2\varrho(\mo-\mo')^2 + 2C^2 C_2 \rho(\mo-\mo',x-x')^2)^{1/2} \le (2 + 2C^2 C_2)^{1/2}\rho(\mo-\mo',x-x').
\end{equation*} 
Similarly, using \eqref{eq:bIneq}, \eqref{eq:sigIneq},  and $\rVert\tau-\tau'\rVert_{L^1(\Q)} \le \rVert\tau-\tau'\rVert_{L^2(\Q)}$, we have 
\begin{equation*}
    \E^{\Q}[|X_{\tau}-X_{\tau'}|] \le C_2 \varrho(\mo-\mo',x-x')+ (C_2 \varrho(\mo-\mo',x-x'))^{1/2}.
\end{equation*}
From $\rho(\mo-\mo',x-x')\le 1$ and the inequality $\rho\le \rho^{1/2}$,  $\rho\in [0,1]$, we have thus shown the existence of a constant $\hat{c}_{g}$ that depends on $C_1,C_2,$ and $C = \lVert Z_* \rVert_{L^2({\Q})}$ such  that
\begin{align*}
    \E^{\Q}[|g(\cO_{\tau},X_{\tau}) - g(\cO_{\tau'}',X_{\tau'}')|]
   & \le c_*\left(\E^{\Q}[\rho(\cO_{\tau}-\cO_{\tau'},X_{\tau}-X_{\tau'})] + \E^{\Q}[\rho(\cO_{\tau'}-\cO_{\tau'}',X_{\tau'}-X_{\tau'}')] \right) \\ 
   &\le \hat{c}_{g}\  \rho(\mo-\mo',x-x')^{1/2}.
\end{align*}
    
\subsection{Running Cost} 
Suppose that $\tau' \le \tau $. Then, the growth and Lipschitz property of $\ell,\ell'$ implies that 
\begin{align*}
    |\int_0^{\tau}\ell_sds - \int_0^{\tau'}\ell_s' ds| &\le   \int_0^{\tau' }|\ell_s - \ell_s'|ds
    + \int_{\tau' }^{\tau}|\ell_s|ds \\ 
    &\le c_*T_* \sup_{s\le  \tau\wedge \tau '}\rho(\cO_{s}-\cO_{s}',X_{s}-X_{s}') +  Z_*' |\tau'- \tau| ,
\end{align*}
where $Z_*' =c_*\sup_{t\le T_*}(1+\rho(\cO_t,X_t)) $. 
The case $\tau\le \tau'$ follows analogously, with $Z_*$ in lieu of $Z_*'$. 
Taking expectation, we obtain 
\begin{align*}
        \E^{\Q}[|\int_0^{\tau}\ell_sds - \int_0^{\tau'}\ell_s'ds|] &\le c_*T_*  \E^{\Q}[ \sup_{s\le  \tau\wedge \tau'}\rho(\cO_{s}-\cO_{s}',X_{s}-X_{s}')] +  C \lVert \tau - \tau'\rVert_{L^2(\Q)},
\end{align*}
with  $C = \lVert Z_* \vee  Z_*' \rVert_{L^2({\Q})} $. Defining $\hat{c}_{\ell} = c_*T_*C_1 + C C_2$, we have from \cref{lem:Gronwall,lem:tau} that 
$$\E^{\Q}[|\int_0^{\tau}\ell_sds - \int_0^{\tau'}\ell_s'ds|]\le \hat{c}_{\ell} \rho(\mo-\mo',x-x')\le \hat{c}_{\ell} \rho(\mo-\mo',x-x')^{1/2},$$
which 
 completes the proof of \cref{prop:Lipschitz} with the constant $\hat{c} = \hat{c}_{g}+ \hat{c}_{\ell}$.

\bibliographystyle{abbrvnat}
\bibliography{ref.bib}

\end{document}